\documentclass{amsart}

\usepackage{amssymb, graphicx, amsrefs}
\usepackage{color}%cyan,red;magenta,green;yellow,blue
\usepackage{mathtools}
\usepackage{pgfplots}

\copyrightinfo{2024}{Atsushi Inoue, Jun Masamune and Rados{\l}aw K. Wojciechowski}

\newtheorem{theorem}{Theorem}[section]
\newtheorem{lemma}[theorem]{Lemma}
\newtheorem{proposition}[theorem]{Proposition}
\newtheorem{corollary}[theorem]{Corollary}

\theoremstyle{definition}
\newtheorem{definition}[theorem]{Definition}
\newtheorem{example}[theorem]{Example}

\theoremstyle{remark}
\newtheorem{remark}[theorem]{Remark}

\numberwithin{equation}{section}

\newcommand{\as}[1]{\left\langle #1\right\rangle}

\newcommand{\ov}[1]{\overline{ #1}}
\newcommand{\ow}[1]{\widetilde{ #1}}

\newcommand{\ms}[1]{\mathcal{ #1}}

\newcommand{\N}{\mathbb{N}}
\newcommand{\Z}{\mathbb{Z}}
\newcommand{\R}{\mathbb{R}}

\newcommand{\si}{\sigma}

\newcommand{\Deg}{\textup{Deg}}
\newcommand{\lm}{\lambda}
\newcommand{\pt}{\partial}
\newcommand{\De}{\Delta}
\newcommand{\al}{\alpha}
\newcommand{\Q}{\mathcal{Q}}
\newcommand{\D}{\mathcal{D}}
\newcommand{\F}{\mathcal{F}}
\newcommand{\Cn}{\mathcal{C}}
\newcommand{\tX}{\widetilde{X}}
\newcommand{\tb}{\widetilde{b}}
\newcommand{\tm}{\widetilde{m}}
\newcommand{\tDeg}{\widetilde{\textup{Deg}}}

\providecommand{\eat}[1]{}

\newcommand{\Hm}[1]{\leavevmode{\marginpar{\tiny%
$\hbox to 0mm{\hspace*{-0.5mm}$\leftarrow$\hss}%
\vcenter{\vrule depth 0.1mm height 0.1mm width \the\marginparwidth}%
\hbox to 0mm{\hss$\rightarrow$\hspace*{-0.5mm}}$\\\relax\raggedright
#1}}}

\begin{document}

\title[Stability and characterizations]{Essential self-adjointness of the Laplacian on weighted graphs: harmonic functions, stability, characterizations and capacity}

\author{Atsushi Inoue}
\address{A.~Inoue, Department of Mathematics, Faculty of Science, Hokkaido University
Kita 10, Nishi 8, Kita-Ku, Sapporo, Hokkaido, 060-0810, Japan}
\email{nanakorobiyaoki@eis.hokudai.ac.jp}

\author{Sean Ku}
\address{S.~Ku, New York University, Courant Institute of Mathematical 
Sciences, 
251 Mercer St, New York, NY 10012}
\email{sk8980@nyu.edu}

\author{Jun Masamune}
\address{J.~Masamune, Mathematical Institute, Tohoku University,
6-3, Aramaki Aza-Aoba, Aoba-ku, Sendai 980-8578,
Graduate School of Science, Tohoku University, Japan}
\email{jun.masamune.c3@tohoku.ac.jp}

\author{Rados{\l}aw K. Wojciechowski}
\address{R.~K.~Wojciechowski, Graduate Center of the City University of New York, 365 Fifth Avenue, New York, NY, 10016.}
\address{York College of the City University of New York, 94-20 Guy R. Brewer Blvd., Jamaica, NY 11451.}
\email{rwojciechowski@gc.cuny.edu}

%\subjclass[2020]{Primary 39A12; Secondary 58J35}

\date{\today}
\thanks{A.~I.~is supported by the Quint Corporation.}
\thanks{S.~K.~is supported by the Queens Experiences in 
Discrete Mathematics (QED) REU program funded by the 
National Science Foundation, Award Number DMS 2150251.}
\thanks{J.~M.~is supported in part by 
JSPS KAKENHI Grant Number 23H03798 and LUPICIA CO., LTD}
\thanks{R.~W.~is supported by PSC-CUNY Awards, jointly funded by the Professional Staff Congress and the City University of New York, and 
the Travel Support for Mathematicians program, funded by the Simons Foundation.}

\begin{abstract}
We give two characterizations for the essential self-adjointness of the weighted Laplacian
on birth-death chains. The first involves the edge weights and vertex measure and is classically known; however,
we give another proof using stability results, limit point-limit circle theory and the connection
between essential self-adjointness and harmonic functions. 
The second characterization involves a new notion of capacity.
Furthermore, we also analyze the essential self-adjointness of Schr{\"o}dinger operators,
use the characterizations for birth-death chains 
and stability results to characterize essential self-adjointness for star-like graphs,
and give some connections to the $\ell^2$-Liouville property.
\end{abstract} 

\maketitle
\tableofcontents

\section{Introduction}
The study of the essential self-adjointness of a variety of operators on weighted graphs is an active area
of research. This includes the case of adjacency operators \cite{Gol10, MO84, Mul87}, weighted Laplacians 
\cite{BG15, Gol14, HKLW12, HMW21, HKMW13, Jor08, JP11, KL12, Mas09, Tor10, Web10, Woj08},
magnetic and Schr{\"o}dinger operators \cite{CTT11, CTT11b, Dod06, GKS16, GS11, Mil11, Mil12, Mil13, MT14, Sch20},
operators on vector bundles \cite{LSW21, GMT14, MT15}, higher dimensional chains \cite{AACT23, AT15, BBJ20, BGJ19, Che18},
and quantum graphs \cite{BK13, Car98, Hae17, EKMN18, KN21, KMN22, KMNb22, KN23a, KN23b}.
We also mention the vast literature on Jacobi matrices where the question of essential self-adjointness appears and
which can be applied to graphs over the integers, see the textbook \cite{Tes00}. 

The large body of literature mentioned above does not seem to yield a characterization for the essential self-adjointness of the Laplacian
in the case of birth-death chains,
i.e., graphs over the natural numbers where only subsequent numbers are connected by weighted edges.
This characterization is classically known by a result of Hamburger \cite{Ham20a, Ham20b} 
due to the connection between the moment problem and essential self-adjointness; see \cite{Ak65} for the moment
problem and \cite{EK18} for the connection between the moment problem and essential self-adjointness. %\Hm{check references}

One of the aims of this note is to give an elementary proof of this result
while introducing techniques that can be applied more generally.
More specifically, we use the stability 
of essential self-adjointness, limit point-limit circle theory
and a careful analysis of harmonic functions on a birth-death chain doubled
to recapture this characterization.
We also place this result in the context of characterizations for other properties, 
e.g., those for recurrence, Markov uniqueness and the Feller property.
%Despite this huge body of literature, in easiest case of  birth-death chains,
%i.e., graphs over the natural numbers where only subsequent numbers are connected by weighted edges, 
%the essential self-adjointness of the Laplacian has not
%been characterized thus far. This characterization is the main aim of this note.

\eat{
In order to give this elementary proof of the Hamburger result
we need three general ingredients.
One ingredient is a stability result.
Roughly speaking, this states that essential self-adjointness is a property on the ends of a graph. In particular,
if two subgraphs are not connected too strongly, then essential self-adjointness of the Laplacian on the entire graph
is equivalent to the essential self-adjointness of the Laplacian on both of the subgraphs. On an abstract level,
this follows directly from perturbation theory via the Kato--Rellich theorem, see \cite{RS75}. Our aim here is to give some
specific geometric conditions on graphs that allow for such a decomposition in the spirit of~\cite{GS11,CTT11}.

The second ingredient for the proof
is the relationship between essential self-adjointness and harmonic functions.
More specifically, for positive operators, there is a connection between essential self-adjointness and the triviality
of $\lambda$-harmonic functions for $\lambda<0$ which are square summable, see \cite{RS75} for general theory 
and \cite{KL12, HKLW12, KLW21} for operators on graphs. 
However, in certain situations one can reduce
this to the study of harmonic functions, i.e., the case $\lambda=0$ \cite{GM11, HMW21}. In our setting,
this makes the analysis of such solutions significantly easier to grasp. 
In particular, we can give some conditions
for harmonic functions to be square summable which is used in the characterization.

The third and final ingredient for the proof is limit point-limit circle theory. This theory discusses essential self-adjointness
for Schr{\"o}dinger operators over the integers. In particular, to establish essential self-adjointness for the Laplacian
on such a graph, it suffices to find a harmonic function which is not square summable at each end, see \cite{Tes00} for more details.

With these three ingredients we proceed as follows:
By the stability result, the essential self-adjointness of the Laplacian on a birth-death chain is equivalent to the 
essential self-adjointness of the Laplacian on a birth-death chain doubled to give a graph over the integers
where the two ends are identical. For birth-death chains, we can characterize square summability of
the positive increasing functions which are harmonic at all but one vertex and then extend
them to be harmonic everywhere on the birth-death chain doubled.
We then use the connection between harmonic functions and essential self-adjointness to prove the characterization
of the essential self-adjointness of the Laplacian on the initial birth-death chain.
%and limit point-limit circle theory, 
%we then give a characterization for the essential self-adjointness of the Laplacian
%on this doubled birth-death chain by using our summability criterion for harmonic functions, 
We note that without the initial doubling procedure and passage to the weighted integers outlined above, 
this approach would not work as all harmonic functions on birth-death chains are constant.}

We then apply this characterization in two ways. The first involves a new notion of capacity which has an additional term involving the norm of the Laplacian and is 
applied to a 
point at infinity. It turns out that,
on birth-death chains, the Laplacian is essentially self-adjoint if and only if this capacity is either zero
or infinite. For related results on manifolds and metric measure spaces see \cite{HKM17, HMS23}; 
however, in these references the capacity is that of a set
removed from the inside of the manifold or metric measure space.
For some connections between another capacity that does not involve
the Laplacian term and Markov
uniqueness see \cite{HKMW13, Mas99, Mas05}.
However, to the best of our knowledge, this is the first time that this
notion of capacity has appeared in the context of graphs and also the
first connection between essential self-adjointness and capacity.

The second application involves so-called star-like graphs, see \cite{CTT11}. These are graphs which, following
the removal of a set, consists of a disjoint union of birth-death chains.
In \cite{CTT11} the removed set is finite.
Here, we significantly generalize this notion to allow the removal
of infinite sets and
use the stability results and the characterizations for birth-death chains
to characterize the essential self-adjointness of the Laplacian on such graphs in terms of both
the graph structure and capacity.

Along the way we have a closer look at conditions needed for the stability
of essential self-adjointness as well as the case of Schr{\"o}dinger operators.
In particular,
we show that a Laplacian on a general graph can be made essentially self-adjoint
by adding a potential or by changing the geometry of the graph.

Finally, we discuss further connections between 
essential self-adjointness
and the $\ell^2$-Liouville property for star-like graphs.
In particular, for a star with two rays, i.e., a graph over the integers
where only subsequent integers are connected by edges,
we give a full characterization of the $\ell^2$-Liouville property.
We achieve this characterization by analyzing the possible end behavior of
all harmonic functions and breaking the problem down into transient
and recurrent cases. 

The paper is organized as follows: In Section~\ref{s:setting} we introduce the setting and recall some general
background results. In Section~\ref{s:stability} we discuss the stability results that we will need 
for the characterizations.
In Section~\ref{s:characterizations} we first prove the characterizations for birth-death chains
in terms of the graph structure and in terms of capacity.
We then extend these characterizations to the case of star-like graphs
and discuss some further connections to Liouville properties.

%%%%%%%%%%%%%%%%%%%%%%%%%%%%%%%%%%%%%%%%%%%%%%%%%%%%%%%%%%%%%

\section{Setting and background results}\label{s:setting}
In this section we introduce the general setting of weighted graphs and Laplacians. 
For a more thorough introduction and proofs of all mentioned results see \cite{KL12, KLW21}.
We then recall some general abstract conditions for establishing essential self-adjointness. 

\subsection{Weighted graphs}
We start by briefly describing our general setting before we focus on special 
cases of interest.
We let $X$ denote a countably infinite set whose elements we refer to as \emph{vertices}. We let
$b\colon X \times X \longrightarrow [0,\infty)$ denote a symmetric function which is 0 on the diagonal and which satisfies
$\sum_{y \in X} b(x,y)<\infty$ for every $x \in X$. If $b(x,y)>0$, then we think of the vertices $x$ and $y$ as being
\emph{connected} by an edge with weight $b(x,y)$, write $x \sim y$, 
and refer to the vertices as \emph{neighbors}.
We refer to the function $b$ as the \emph{edge weight} as it gives the edge 
structure of the graph. 
We call any sequence of vertices $(x_k)$ such that $x_k \sim x_{k+1}$ 
a \emph{path} from the first vertex to the last vertex in the sequence. 
If there is a path connecting any two vertices in the graph,
then we call the graph \emph{connected}.

Furthermore, we let 
$m \colon X \longrightarrow (0,\infty)$
denote a strictly positive \emph{vertex measure} which we extend to all subsets of $X$ by additivity. We refer
to the triple $(X, b, m)$ as a \emph{weighted graph}.
We say that $(X,b,m)$ is an \emph{induced subgraph} of $(\tX,\tb,\tm)$, equivalently,
that $(\tX,\tb,\tm)$ is an \emph{induced supergraph} of $(X,b,m)$,
if $X \subseteq \tX$, $\tb \vert_{X \times X}=b$ and $\tm\vert_{X}=m$.

Throughout the paper, 
we will make an additional assumption on the relationship between the edge weight and the vertex measure, namely, we assume that
\begin{equation*}\label{A}
\sum_{y \in X} \frac{b^2(x,y)}{m(y)}<\infty \qquad \textup{for every } x \in X.
\tag{A}
\end{equation*}
As we will explain below, this condition is needed in order for the formal Laplacian to map
the finitely supported functions to those which are square summable with respect to $m$.
It is trivially satisfied if the graph is \emph{locally finite}, that is, 
if every vertex has finitely many neighbors. More generally, it is also satisfied if $\inf_{y \sim x} m(y) >0$
for all $x \in X$.

Perhaps the easiest classes of graphs consists of birth-death chains
and stars with two rays. In this case, the vertex set is either the natural numbers 
or the integers
and subsequent numbers are the only neighbors.
\begin{definition}[Birth-death chains and stars with two rays]
A connected weighted graph is a \emph{birth-death chain} if
$X \subseteq \N_0=\{0, 1, 2, \ldots\}$ and
$b(x,y)>0$ if and only if $|x-y|=1$.
A weighted graph is a \emph{star with two rays} if $X = \Z=\{\ldots, -2, -1, 0, 1, 2, \ldots\}$ and
$b(x,y)>0$ if and only if $|x-y|=1$.
\end{definition}
We note that
as each vertex has either one or two neighbors, it is clear that birth-death
chains and stars with two rays satisfy (A).

\subsection{The energy form, Laplacians, restrictions, adjoints and essential self-adjointness}
We next introduce the energy form and corresponding 
Laplacian. We let $C(X) = \{ f \colon X \longrightarrow \R\}$
denote the set of all real-valued functions on $X$, 
$$C_c(X) = \{ f \in C(X) \mid f(x) \neq 0 \textup{ for finitely many } x \in X\}$$
denote the set of finitely supported functions on $X$ and 
$$\ell^2(X,m) = \{ f \in C(X) \mid \sum_{x \in X} f^2(x)m(x)<\infty\}$$
denote the Hilbert space of square summable functions with respect to $m$ on $X$. We define an inner product on $\ell^2(X,m)$
by $\as{f,g} = \sum_{x \in X} f(x)g(x)m(x).$

Next, we define the energy of a function. Namely, for $f \in C(X)$, we let
$$\Q(f) = \frac{1}{2}\sum_{x,y \in X} b(x,y) (f(x)-f(y))^2 $$
denote the \emph{energy} of $f$. We then define
$$\D = \{ f \in C(X) \mid \Q(f) < \infty\}$$
as the set of \emph{functions of finite energy}. For $f, g \in \D$, we let
$$\Q(f,g) = \frac{1}{2} \sum_{x,y \in X} b(x,y)(f(x)-f(y))(g(x)-g(y)).$$
%For a birth-death chain we note that
%$$\Q(f)= \sum_{r=0}^\infty b(r,r+1) (f(r)-f(r+1))^2.$$

In order to introduce a formal Laplacian, we let 
$$\ms{F} = \{ f \in C(X) \mid \sum_{y \in X} b(x,y) |f(y)| <\infty \textup{ for every } x \in X \}$$
and for $f \in \ms{F}$ and $x \in X$, we let
$$\De f(x) = \frac{1}{m(x)} \sum_{y \in X} b(x,y)(f(x)-f(y)).$$
A direct calculation gives that assumption (A), i.e., that $\sum_{y \in X} b^2(x,y)/m(y)<\infty$ for every $x \in X$
is equivalent to
$\De(C_c(X)) \subseteq \ell^2(X,m)$, that is, that the formal Laplacian maps
finitely supported functions to square summable functions. Thus, (A) implies
that
$$L_c = \De \vert_{C_c(X)},$$
the restriction of the formal Laplacian to the finitely supported functions, gives an operator on $\ell^2(X,m)$. Furthermore, assumption (A)
is equivalent to the fact that $\ell^2(X,m) \subseteq \F$.

A direct calculation then yields the following Green's formula
$$\Q(\varphi, \phi)= \as{L_c \varphi, \phi} = \as{\varphi, L_c \phi}$$
for all $\varphi, \phi \in C_c(X)$ so that $L_c$ is a symmetric positive operator.
We say that $L_c$ is \emph{essentially self-adjoint} whenever $L_c$ has a unique self-adjoint extension.
By general theory this is equivalent to $\ov{L}_c$, the closure of $L_c$, being self-adjoint, see \cite{RS80, RS75, Wei80}
for general background results on essential self-adjointness.

For a vertex $x \in X$, we let
$$\Deg(x) = \frac{1}{m(x)} \sum_{y \in X} b(x,y)$$
denote the \emph{weighted vertex degree} of $x$. In particular, $\Deg$ is a bounded
function on $X$ if and only if $L_c$ is a bounded operator on $\ell^2(X,m)$. %In this case, $L_c$ is obviously essentially self-adjoint.

We denote the adjoint of $L_c$ by $L_c^*$. By general theory, it can be shown that $L_c^*$ is a restriction of $\De$
and has domain
$$D(L_c^*) = \{ f \in \ell^2(X,m) \mid \De f \in \ell^2(X,m) \}.$$
In the next few subsections, we will discuss some criteria for the essential self-adjointness of $L_c$ via properties
of the adjoint $L_c^*$.

\subsection{$\lambda$-harmonic functions}
The first criterion for essential self-adjointness involves functions which are in the kernel of $L_c^*$. 
Let $\lm \in \R$. We say that $f \in \ms{F}$ is $\lm$-\emph{harmonic} if
$$\De f=\lm f.$$ 
When $\lm=0$, i.e., when $\De f=0$, we say that $f$ is \emph{harmonic}.

By general theory, the essential self-adjointness of $L_c$ is equivalent to the triviality of $\lm$-harmonic
functions which are in $\ell^2(X,m)$ for $\lm<0$. 
In other words, $L_c$ is essentially self-adjoint if and only if
$$\ker(L_c^*-\lm) =\{0\}$$
for one (all) $\lm<0$, see \cite{KL12, HKLW12, RS75} for details. We further note that the 
essential self-adjointness of $L_c$ implies
that harmonic functions which are in $\ell^2(X,m)$ are constant, that is, that the 
\emph{$\ell^2$-Liouville property} holds. In other words, if 
$\ker(L_c^*-\lm) =\{0\}$, then
$\ker(L_c^*)$ consists of only constant functions, see \cite{HMW21}. 
Equivalently, if there exists a non-constant harmonic function in $\ell^2(X,m)$, then $L_c$
is not essentially self-adjoint. We will use this fact later.

\subsection{Limit point-limit circle theory}
We next discuss some aspects of limit point-limit circle theory. 
For this, we focus on stars with two rays, i.e., 
$X=\Z=\{\ldots, -2, -1, 0, 1, 2, \ldots\}$ and 
$b(x,y)>0$ if and only if $|x-y|=1$.

We say that a function $f$ on $\Z$ is \emph{in $\ell^2$ at $\infty$}, respectively 
\emph{at $-\infty$}, if
$$\sum_{r=0}^\infty f^2(r)m(r)<\infty, \qquad \textup{ resp. } \sum_{r=0}^\infty f^2(-r)m(-r)<\infty.$$
We write $f \in \ell^2(\infty)$, respectively
$f \in \ell^2(-\infty)$, in this case. 
Otherwise, we write $f \not \in \ell^2(\infty)$, respectively 
$f \not \in \ell^2(-\infty)$. 

For stars with two rays,
it turns out that if for one $\lm \in \R$ there exists a $\lm$-harmonic function 
$f \not \in \ell^2(\infty)$ (resp.~$f \not \in \ell^2(-\infty)$), 
then for all $\lm \in \R$ there exists a $\lm$-harmonic function 
$f \not \in \ell^2(\infty)$ (resp.~$f \not \in \ell^2(-\infty)$). 
In this case, we say that $L_c$ is in 
the \emph{limit point case at $\infty$}, respectively \emph{at $-\infty$}. Otherwise, $L_c$ is said to be in the \emph{limit circle
case at $\infty$}, respectively \emph{at $-\infty$}.

Now, a main result in this setting is that $L_c$ is essentially self-adjoint if and only if
$L_c$ is in the limit point case at both $\infty$ and $-\infty$. In particular, to establish the essential 
self-adjointness of $L_c$, it suffices to find a harmonic function which is not in $\ell^2$ at $\infty$
and a harmonic function which is not in $\ell^2$ at $-\infty$. For further details and proofs 
see \cite{Tes00}. 

\subsection{Symmetry of the adjoint}
A final general criterion for essential self-adjointness is that $L_c^*$ is symmetric. To state this we introduce
some basic notation that we will use further below.
For $f, g \in D(L_c^*)$ we write
$$[f, g]_{L_c^*}=\as{L_c^*f, g}-\as{f, L_c^*g}.$$
We now state a basic fact which follows easily from general theory,
see e.g.~\cite{RS75, RS80, Wei80}.
For the convenience of the reader, we give a short proof.
\begin{lemma}\label{l:basic fact}
Let $(X,b,m)$ be a weighted graph satisfying (\ref{A}).
Then, the following statements are equivalent:
\begin{enumerate}
\item[(i)] $L_c$ is essentially self-adjoint.
\item[(ii)] $L_c^*$ is symmetric.
\item[(iii)] $[f, g]_{L_c^*}=0$ for all $f, g \in D(L_c^*)$.
\end{enumerate}
\end{lemma}
\begin{proof}
It is clear that (ii) is equivalent to (iii) by definition. To see
that (i) implies (ii) we note that essential self-adjointness of $L_c$
is equivalent to the self-adjointness of $\ov{L}_c=L_c^{**}$ 
and that $L_c^*$ is always a closed operator. Hence, $L_c \subseteq L_c^*= \ov{L_c^*} = \ov{L}_c^*=\ov{L}_c=L_c^{**}$
and thus $L_c^*$ is symmetric.
Now, assume that (ii) holds. 
Then we obtain $\ov{L}_c \subseteq L_c^* \subseteq L_c^{**}=\ov{L}_c$, which implies that $\ov{L}_c$ is self-adjoint and thus $L_c$ is 
essentially self-adjoint.
\end{proof}

%%%%%%%%%%%%%%%%%%%%%%%%%%%%%%%%%%%%%%%%%%%%%%%%%%%%%%%%%%%%%%%%%%%

\section{Stability}\label{s:stability}
We now discuss stability results. In particular, in order to establish the essential self-adjointness
of the Laplacian on an entire graph, it suffices to decompose the graph into two subgraphs which are not too 
strongly connected and then analyze the essential self-adjointness of the Laplacian on each of the subgraphs. 

Such results follow directly from the Kato--Rellich theorem. 
However, our focus is on giving
some geometric conditions on the underlying graph to ensure that such
a decomposition is possible.
For the case when one of the subgraphs is finite see \cite{CTT11}. Here, we rather replace the finiteness assumption
with some bounded degree assumptions in the spirit of stability results for stochastic completeness found in \cite{Hua11, KL12}.
We will later show the necessity of such assumptions.

\subsection{Decomposing the graph}
We start with the general notion of a vertex boundary and interior of a subset. 
For $K \subseteq X$, let
$$\pt K = \{ x \in K \mid \textup{ there exists } y \not \in K \textup{ such that } b(x,y)>0 \}$$
denote the \emph{vertex boundary} of the subset and let
$\mathring{K} = K \setminus \pt K$ denote the \emph{interior} of $K$.
%In other words, for a subset $K$ the boundary consists of those vertices in $K$ which have neighbors
%not in $K$ and the interior those for which all neighbors are in $K$.

We let $X_1 \subseteq X$, $X_2 = X \setminus X_1$ % denote the complement of $X_1$ in $X$. 
and $X_3 = \pt X_1 \cup \pt X_2.$
With these three pieces in mind, we now construct three subgraphs of our original graph.
For the first two subgraphs
we take the subgraphs induced by $X_1$ and $X_2$.
Specifically, we let $b_k \colon X_k \times X_k \longrightarrow [0,\infty)$
and $m_k \colon X_k \longrightarrow (0,\infty)$ denote the restrictions of $b$ to $X_k \times X_k$ 
and $m$ to $X_k$ for $k=1, 2$. This gives the induced subgraphs $(X_1, b_1, m_1)$ and $(X_2, b_2, m_2)$.

For $X_3$, we do not restrict $b$ but rather only consider the edges that connect 
$X_1$ and $X_2$. That is, we define $b_\partial$ on $X_3 \times X_3$ via
$$b_\partial(x, y)=
\begin{cases}
b(x, y) \quad & \textup{if } x \in \partial X_1, y \in X_2 \\
b(x, y) \quad & \textup{if } x \in \partial X_2, y \in X_1 \\
0 \quad &\text{otherwise}.
\end{cases}
$$
For the vertex measure on $X_3$ we restrict $m$ to obtain $m_3$. This gives the subgraph
$(X_3, b_\partial, m_3)$.
We note that all edges in the original graph fall under the purview of exactly one of the 
edge weight functions. In other words, if we extend $b_1, b_2$ and $b_\partial$ to be zero outside of
their domains of definition, then we get the equality
$b=b_1+b_2+b_\partial$ 
on $X \times X$. 

In line with the weighted vertex degree introduced previously, we will denote the weighted boundary vertex degree function as
$$\Deg_\partial(x) = \frac{1}{m(x)} \sum_{y \in X_3}b_\partial(x,y)$$
for $x \in X_3$. We note that letting
$\Deg_{\pt X_k}(x) = 1/m(x) \sum_{y \not \in X_k}b(x,y)$ for $x \in X_k$ with $k=1, 2$ and extending by 0
we get that $\Deg_\partial = \Deg_{\pt X_1}+\Deg_{\pt X_2}$.

\subsection{Decomposing the Laplacian}
Given the decomposition of $b$ presented above, we now give a corresponding decomposition
for the Laplacian. More specifically, we let $\De_k$ denote the formal Laplacian of the induced
subgraphs $(X_k,b_k,m_k)$ for $k=1,2$ and let $\De_\partial$ denote the formal Laplacian 
of the graph $(X_3, b_\partial, m_3)$. We then
let $L_{c,k}$ denote the restriction of $\De_k$ to $C_c(X_k)$ for $k=1,2$ and $L_{c,\partial}$
the restriction of $\De_\partial$ to $C_c(X_3)$. 
In particular, we note that as $b_1$ will not give any connections to vertices in $X_2$ and $b_2$ will not give
any connections to vertices in $X_1$,
the Laplacians $L_{c,1}$ and $L_{c,2}$ can be thought of as having Neumann boundary conditions on $\pt X_1$
and $\pt X_2$, respectively. 
%The Laplacian $L_{c,\partial}$ sees the edges that go between $X_1$ and $X_2$.

We note that the Laplacians $\De_k$ and $\De_\partial$ are only defined on appropriate subspaces of $C(X_k)$.
In the formula below, we wish to apply these to functions which are defined on $X$. Thus, whenever
we restrict a function $f \in C(X)$ to $X_k$ we write $f_k$, i.e., 
$f_k = f \vert_{X_k}$. Furthermore, we note that $\De_k f_k$
is defined only on $X_k$ for $k=1,2$ and $\De_\partial f_3$ is defined on $X_3$;
however, we extend by $0$ for vertices that are outside of $X_k$. 
In particular, we can think
of $\De_k$ and $\De_\partial$ as being defined on $\ms{F}$, the domain of the formal Laplacian $\De$. 
With these preparations, we then get the following general decomposition result.
\begin{lemma}\label{l:decomposition}
For $f \in \ms{F}$ and $x \in X$ we have
$$\De f (x)= \De_1f_1(x)+\De_2f_2(x)+\De_\partial f_3(x).$$
\end{lemma}

\begin{proof}
For $x \in \partial X_1$, we calculate
\begin{align*}
m(x)(\De f)(x)%&=\sum_{y \in X} b(x, y)(f(x)-f(y)) \\
&=\sum_{y \in X_1} b(x, y)(f(x)-f(y))+\sum_{y \in X_2} b(x, y)(f(x)-f(y)) \\
&=\sum_{y \in X_1} b_1(x, y)(f_1(x)-f_1(y))+\sum_{y \in X_2} b_\partial(x, y)(f_3(x)-f_3(y)) \\
&=m(x)\left(\De_1 f_1(x)+\De_\partial f_3(x)\right).
\end{align*}
For $x \in \mathring{X}_1$, we obtain
\begin{align*}
m(x)(\De f)(x)&=\sum_{y \in X_1} b(x, y)(f(x)-f(y))
=m(x)(\De_1 f_1)(x).
\end{align*}
Thus, we find that $\De f=\De_1 f_1+\De_\partial f_3= \De_1 f_1+ \De_2 f_2+ 
\De_\partial f_3 $ on $ X_1$. Similarly, we can calculate
that $\De f=\De_2 f_2+\De_\partial f_3$ on $X_2$.
This completes the proof.
\end{proof}

Next we consider the restrictions of the formal Laplacians to the finitely supported functions.
We recall that $D(L_c^*) = \{ f \in \ell^2(X,m) \mid \De f \in \ell^2(X,m) \}$ and, similarly, 
$$D(L_{c,k}^*) = \{ f \in \ell^2(X_k,m_k) \mid \De_k f \in \ell^2(X_k,m_k) \}$$
for $k=1, 2$ with an analogous statement for $D(L_{c,\partial})$, see \cite{KL12}. 
%We also recall the notation $[f, g]_{L_c}=\as{L_c^*f, g}-\as{f, L_c^*g}$ for $f, g \in D(L_c^*)$
%and that $L_c$ is essentially self-adjoint if and only if $[f, g]_{L_c}=0$ for all $f, g \in D(L_c^*)$
%by Lemma~\ref{l:basic fact}.

As above, we will restrict functions defined on $X$ to $X_k$ and then
extend by 0. However, one has to be careful about when the restriction are in the appropriate
domains. Thus we make an additional assumption that will always ensure this.
More specifically, we assume that the restriction of the formal Laplacian $\De_\partial$ to $\ell^2(X_3,m_3)$ is 
a bounded operator. This is equivalent to the fact that $\Deg_\partial$ is bounded on $X_3$, see \cite{KL12, KLW21}.

\begin{lemma}\label{l:stab_key_lemma}
Suppose that $\Deg_\partial$ is bounded on $X_3$. Then, we have:
\begin{itemize}
\item[(a)] If $f \in D(L_c^*)$, then $f_k \in D(L_{c,k}^*)$ for $k=1, 2$.
\item[(b)] If $f, g \in D(L_c^*)$, then
$$[f, g]_{L_c^*}=\sum_{k=1}^2 [f_k, g_k]_{L_{c,k}^*}.$$
\end{itemize}
\end{lemma}

\begin{proof}
(a): We shall carry out the proof for $k=1$. 
It is clear that $f_1 \in \ell^2(X_1, m_1)$.
By Lemma~\ref{l:decomposition}, we obtain 
$$\De_1 f_1=\De f-\De_\partial f_3$$ 
on $X_1$ so that
$\De_1 f_1 \in \ell^2(X_1, m_1)$ as $\De_\partial$ is bounded on $\ell^2(X_3, m_3)$ 
and $\De f \in \ell^2(X,m)$ since $f \in D(L_c^*)$.
Thus, $f_1 \in D(L_{c,1}^*)$. The same proof works for $k=2$.

(b): By Lemma~\ref{l:decomposition}, for $f, g \in D(L_c^*)$ we obtain
$$ \as{\De f, g} = \sum_{k=1}^2\as{\De_k f_k, g} + \as{\De_\partial f_3, g}  $$
which clearly implies $[f, g]_{L_c^*}=\sum_{k=1}^2[f, g]_{L_{c,k}^*}+ [f, g]_{L_{c,\partial}^*} $ as all adjoints are restrictions
of the corresponding formal Laplacian.
Furthermore, as $\De_\partial$ gives a bounded operator on $\ell^2(X_3,m_3)$, we obtain 
$[f_3, g_3]_{L_{c,\partial}^*}=0$.
This completes the proof.
\end{proof}

\subsection{Basic stability result}
We now have all of the ingredients to prove the following stability result.
\begin{theorem}\label{t:stability}
Let $(X, b, m)$ be a weighted graph satisfying (\ref{A}).
If $X_1 \subseteq X, X_2 =X \setminus X_1$ and $X_3 = \pt X_1 \cup \pt X_2$
are such that $\Deg_\partial$ is a bounded function on $X_3$, then
$$L_c \textup{ is essentially self-adjoint} \iff  L_{c,1} \textup{ and } L_{c,2} \textup{ are essentially self-adjoint.}$$
\end{theorem}
\begin{proof}
$\Longleftarrow$: Assume that $L_{c,1}$ and $L_{c,2}$ are essentially self-adjoint. 
Let $f, g \in D(L_c^*)$. By (a) in Lemma~\ref{l:stab_key_lemma}, 
$f_k, g_k \in D(L_{c,k}^*)$ and $[f_k, g_k]_{L_{c,k}^*}=0$ for $k=1, 2$ by Lemma~\ref{l:basic fact} as we assume essential self-adjointness. 
Thus, from (b) in Lemma~\ref{l:stab_key_lemma}, we obtain
$$[f, g]_{L_c^*}=\sum_{k=1}^2 [f_k, g_k]_{L_{c,k}^*}=0$$
which implies that $L_c$ is self-adjoint by Lemma~\ref{l:basic fact}.

$\Longrightarrow$: Assume that $L_{c,1}$ is not essentially self-adjoint. 
Then, by Lemma~\ref{l:basic fact}, there exist $f, g \in D(L_{c,1}^*)$ such that $[f, g]_{L_{c,1}^*} \neq 0$. 
We define $\widetilde{f}, \widetilde{g} \in \ell^2(X, m)$ by extending $f$ and $g$ by 0 outside of $X_1$.
By Lemma~\ref{l:decomposition} we obtain
\begin{align*}
(\De \widetilde{f})(x)&=(\De_1 \widetilde{f}_1)(x)+(\De_\partial \widetilde{f}_3)(x)=(\De_1 f)(x)+(\De_\partial \widetilde{f}_3)(x) \quad \textup{ for } x \in X_1, \\
(\De \widetilde{f})(x)&=(\De_2 \widetilde{f}_2)(x)+(\De_\partial \widetilde{f}_3)(x)=(\De_\partial \widetilde{f}_3)(x)  \ \qquad \qquad \qquad \textup{ for } x \in X_2
\end{align*}
with analogous formulas for $\ow{g}$. Therefore,
we obtain $\ow{f}, \ow{g} \in D(L_c^*)$ since $\De_\partial$ is bounded and $f, g \in D(L_{c,1}^*)$. 
From (b) in Lemma~\ref{l:stab_key_lemma} and the fact that
$[\widetilde{f}_2, \widetilde{g}_2]_{L_{c,2}^*}=0$, we then obtain
$$[\widetilde{f}, \widetilde{g}]_{L_c^*}=\sum_{k=1}^2[\widetilde{f}_k, \widetilde{g}_k]_{L_{c,k}^*}=[f, g]_{L_{c,1}^*} \neq 0.$$
Therefore, $L_c$ is not essentially self-adjoint by Lemma~\ref{l:basic fact}.
An analogous argument works if $L_{c,2}$ is not essentially self-adjoint. This completes the proof.
\end{proof}

\begin{remark}[Kato--Rellich]
The result above can be strengthened by applying the Kato--Rellich theorem.
More specifically, we consider $L_{c,k}$ to be operators on $C_c(X)$ by extending them by 0 outside of $C_c(X_k)$ and use the fact
that $L_{c,1}+L_{c,2}$ is essentially self-adjoint if and only if both $L_{c,1}$ and $L_{c,2}$ are essentially self-adjoint
which can be obtained by using techniques found in the proof of Theorem~\ref{t:stability}.
By a variant of the Kato--Rellich theorem we then obtain that if there exist $ \alpha, \beta \in \R$ with $ \alpha < 1$ such that
$$ \| L_{c,\partial} \varphi \| \leq \alpha \left(\|L_c \varphi \| + \| (L_{c,1}+L_{c,2}) \varphi \| \right) + \beta \|\varphi\|$$
for all $\varphi \in C_c(X)$, then $L_c$ is essentially self-adjoint if and only if  $L_{c,1}$ and  $L_{c,2}$ are essentially self-adjoint,
see Theorem~X.13 in \cite{RS75}. Theorem~\ref{t:stability} is the special case when $\alpha=0$ as $L_{c,\partial}$ is bounded
if and only if $\Deg_\partial$ is bounded. As the result of Theorem~\ref{t:stability} is sufficient for our purposes, we have given a
proof to make the presentation self-contained.
\end{remark}

%%%%%%%%%%%%%%%%%%%%%%%%%%%%%%%%%%%%%%%%%%%%%%%%%%%%%%%%%%%%%%%%%%%

\section{Characterizations}\label{s:characterizations}
In this section we prove some characterizations for the essential self-adjointness 
of the Laplacian on birth-death chains. The first characterization is in terms of the edge weights and vertex measure.
%The proof of this characterization we give
%will use all of the ingredients discussed above, namely, the stability theory, the limit point-limit circle theory
%and the connection between essential self-adjointness and harmonic functions plus an easy observation concerning
%harmonic functions on birth-death chains.
We then give some applications of this result by connecting
the characterization for essential self-adjointness to other properties
such as form uniqueness and the Feller property. 
We also discuss the case of Schr{\"o}dinger operators on graphs.
There we show that there always exists a potential such that the Laplacian
plus that potential is essentially self-adjoint on locally finite graphs. For birth-death
chains, we show that a potential that is bounded below will not break the 
essential self-adjointness while an unbounded potential may do so.

The second characterization of the essential self-adjointness
of the Laplacian on birth-death chains is in terms of a new notion of capacity.
Here we show that essential self-adjointness is equivalent to the fact that
a point at infinity has zero or infinite capacity.
We will then discuss the necessity
of the assumption for the stability result above by showing that any
graph has an induced supergraph whose Laplacian is essentially self-adjoint.
We then extend our characterizations of essential self-adjointness 
to star-like graphs.
Finally, we characterize the $\ell^2$-Liouville property in the case 
of stars with two rays in the final subsection
of the paper. This includes some further connections to essential self-adjointness.

\subsection{Harmonic functions}
In order to use harmonic functions for our characterization of essential self-adjointness, 
we need to analyze their behavior on birth-death chains. First, we observe
that if a function is harmonic everywhere on a birth-death chain, then it is constant and, thus, will be in $\ell^2(\N_0,m)$
if and only if the measure of the entire vertex set is finite. However, if we relax the assumption on the harmonicity
at the first vertex, then we obtain non-constant functions and
the question if they are in $\ell^2(\N_0,m)$ becomes more interesting.

\begin{lemma}\label{l:harmonic}
Let $(\N_0,b,m)$ be a birth-death chain. Then, $v \in C(\N_0)$ satisfies $\De v(r)=0$ for 
all $r \neq 0$
if and only if
$$v(r+1) = v(1) + C \sum_{k=1}^r \frac{1}{b(k,k+1)}$$
for all $r \neq 0$ where $C=b(0,1)(v(1)-v(0)).$ In particular, if $v(1)\geq0$ and $C>0$, then $v$ is non-constant and $v \in \ell^2(X,m)$
if and only if
$$\sum_{r=0}^\infty \left( \sum_{k=0}^r \frac{1}{b(k,k+1)} \right)^2 m(r+1) < \infty.$$
An analogous statement holds if $v(1)\leq 0$ and $C<0$.
\end{lemma} 
\begin{proof}
The proof of the formula for $v$ is an easy calculation using strong induction. The remaining
statements are obvious consequences of the formula and basic arguments.
\end{proof}

\subsection{Characterization for birth-death chains via edge weights and vertex measure}
We now use all of the ingredients introduced above
to prove the first characterization. As mentioned in the introduction, this result is
classically known, see \cite{Ham20a, Ham20b, Ak65, EK18}.
Thus, we refer to it as Hamburger's criterion.
%The only other ingredient for the proof is that we double the birth-death chain
%to get a graph over the integers in order to apply the limit point-limit circle theory.

\begin{theorem}[Hamburger's criterion]\label{t:bd_characterization}
Let $(\N_0,b,m)$ be a birth-death chain. Then, the following statements are equivalent:
\begin{itemize}
\item[(i)] $L_c$ is essentially self-adjoint.
\item[(ii)] $\displaystyle{\sum_{r=0}^\infty \left( \sum_{k=0}^r \frac{1}{b(k,k+1)} \right)^2 m(r+1) = \infty}.$
\end{itemize}
\end{theorem}
\begin{proof}
Before we prove the equivalence we give a general construction which will
be used for both implications.

We first double the birth-death chain $(\N_0,b,m)$
to get an induced supergraph over the integers with an additional vertex. More specifically, 
we let
$$\ow{X}= -\N \cup \{\ow{0}, 0 \} \cup \N$$
with edge weight $\ow{b} \colon \ow{X} \times \ow{X} \longrightarrow [0,\infty)$ given by
$$\ow{b}(k,k+1)= b(k,k+1)=\ow{b}(-k-1,-k)$$
for $k=1, 2, 3, \ldots$ with $\ow{b}(-1, \ow{0})=b(0,1)=\ow{b}(0,1)$ and extended symmetrically. Furthermore, we
let $\ow{b}(0,\ow{0})=\ow{b}(\ow{0},0)>0$ so that $0$ and $\ow{0}$ are connected and let $\ow{b}$ be zero for all other pairs of vertices.

Similarly, for the vertex measure $\ow{m} \colon \ow{X} \longrightarrow (0,\infty)$ we let
$\ow{m}(k)=\ow{m}(-k)=m(k)$ for $k=1, 2, 3, \ldots$ and let $\ow{m}(\ow{0})=m(0)=\ow{m}(0)$.
That is, we merely double the original birth-death chain and add one additional vertex so
that we get two exact copies of the original birth-death chain when we remove the edge
connecting $0$ and $\ow{0}$. By Theorem~\ref{t:stability}, it follows that $L_c$ is essentially self-adjoint
if and only if the Laplacian arising from $(\ow{X}, \ow{b},\ow{m})$ is essentially self-adjoint.

We now use Lemma~\ref{l:harmonic} to define a harmonic function on $\ow{X}$. More specifically, we let
$v(0)=0$, let $v(1)>0 $ and for $r \geq 1$ we define
$$v(r+1)=v(1) + C \sum_{k=1}^r \frac{1}{b(k,k+1)}$$
where $C=b(0,1)v(1)>0$.
We then let $v(\ow{0})=-C/\ow{b}(0,\ow{0})$ and
$$v(-r)=v(\ow{0})-C \sum_{k=0}^{r-1}\frac{1}{b(k,k+1)}$$
for $r\geq 1$.
By direct calculations and Lemma~\ref{l:harmonic}, $v$ is harmonic on $(\ow{X},\ow{b},\ow{m})$. 

After this general construction, we now give the proof.

\bigskip

(i) $ \Longrightarrow$ (ii):
If 
$$\sum_{r=0}^\infty \left( \sum_{k=0}^r \frac{1}{b(k,k+1)} \right)^2 m(r+1) < \infty,$$
then $v$ is harmonic, non-constant and $v \in \ell^2(\ow{X},\ow{m})$ by Lemma~\ref{l:harmonic}. 
Thus, the Laplacian on $(\ow{X},\ow{b},\ow{m})$ is not essentially self-adjoint as the $\ell^2$-Liouville
property fails, see \cite{HMW21}. Consequently,  
$L_c$ is not essentially self-adjoint
by Theorem~\ref{t:stability}.

\bigskip

(ii) $ \Longrightarrow$ (i): If
$$\sum_{r=0}^\infty \left( \sum_{k=0}^r \frac{1}{b(k,k+1)} \right)^2 m(r+1) = \infty,$$
then $v$ is harmonic on $(\ow{X}, \ow{b}, \ow{m})$ and $v$ is not in $\ell^2$ at $\infty$ and at $-\infty$
by Lemma~\ref{l:harmonic}.
Therefore, the Laplacian on $(\ow{X},\ow{b},\ow{m})$ 
is essentially self-adjoint by the limit point-limit circle theory, see \cite{Tes00}, 
and $L_c$ is essentially self-adjoint by Theorem~\ref{t:stability}.
\end{proof}

\subsection{Connections with other properties}
We now give some connections between essential self-adjointness 
and other properties on birth-death chains.

We first discuss form, or
equivalently, Markov uniqueness. For this, we consider two closed restrictions
of the energy form. Namely, $Q^{(D)}$ acts as $\Q$ on the space
$$D(Q^{(D)}) = \overline{C_c(X)}^{\| \cdot \|_{\Q}}$$
where $\|f\|_{\Q}^2 = \|f\|+\Q(f)$ is the form norm. The second
restriction, $Q^{(N)}$, acts as $\Q$ on
$D(Q^{(N)})=\ell^2(X,m) \cap \D$, the set of functions of finite energy in $\ell^2(X,m)$.

We say that a graph satisfies
\emph{form uniqueness} if $Q^{(D)}=Q^{(N)}$. In this case,
there is a unique closed form associated to the graph, see \cite{HKLW12, KLW21}
for more details. Equivalently, there is a unique
Markov realization of $\Delta$, that is, an operator which acts
as $\Delta$ and whose associated form is a Dirichlet form, see \cite{Sch20, KLW21}.
We note that essential self-adjointness always implies form uniqueness,
see \cite{HKLW12, KLW21}.

We also recall that a connected graph is called \emph{transient} if
$$\int_0^\infty e^{-tL}1_y(x)dt < \infty$$
for some (all) $x, y \in X$ where $e^{-tL}$ is the semigroup
associated to the operator $L=L^{(D)}$ coming from the form $Q^{(D)}$
and $1_y$ is the indicator function of the set $\{y\}$.
Otherwise, a graph is called \emph{recurrent}.

\begin{corollary}[Transient chains]\label{c:transient}
Let $(\N_0,b,m)$ be a transient birth-death chain. Then, the following statements are equivalent:
\item[(i)] $L_c$ is essentially self-adjoint.
\item[(ii)] $Q^{(D)}=Q^{(N)}$.
\item[(iii)] $m(\N_0)=\infty$.
\end{corollary}
\begin{proof}
It is well-known that the transience of a birth-death chain is equivalent to
$$\sum_{k=0}^\infty \frac{1}{b(k,k+1)}<\infty$$
see, e.g., \cite{KLW21, Woe09}. Thus, the equivalence of (i) and (iii) follows
directly from Theorem~\ref{t:bd_characterization}.

That (i) implies (ii) is a general fact which does not require transience
or the birth-death chain structure, see \cite{HKLW12, KLW21}. Finally,
that (ii) implies (iii) follows from the fact that if $m(\N_0)<\infty$,
then transience and failure of form uniqueness are equivalent, see \cite{Sch17}.
\end{proof}

\begin{remark}
For some related results connecting form uniqueness and measure
in the case of quantum graphs see \cite{KN23a}.
\end{remark}

For the recurrent case, we will present a connection to the semigroup 
vanishing at infinity.
More specifically, we let $\|f\|_\infty = \sup_{x \in X} |f(x)|$ and
$$C_0(X) = \overline{C_c(X)}^{\| \cdot \|_\infty}$$
denote the space of functions that vanish at infinity. We say that a graph satisfies
the \emph{Feller property} if
$$e^{-tL}(C_c(X)) \subseteq C_0(X)$$
for some (all) $t \geq 0$, see \cite{Woj17, HMW19, Adr21, KMW} for more details on the Feller property for
graphs.
\begin{corollary}[Recurrent chains]
Let $(\N_0,b,m)$ be a recurrent birth-death chain.
If the chain satisfies the Feller property, then $L_c$ is essentially self-adjoint.
\end{corollary}
\begin{proof}
We let $B_r^c = \{r+1, r+2, r+3, \ldots\}$ denote the complement of the ball
of radius $r$ about $0$ in $\N_0$.
In the recurrent case, i.e., when $\sum_{k} 1/b(k,k+1) = \infty$, 
the Feller property on birth-death chains is equivalent to
$\sum_{r} {m(B_r^c)}/{b(r,r+1)} = \infty$, see \cite{Woj17}.
 Now, rearranging gives
$$\sum_{r=0}^\infty \frac{m(B_r^c)}{b(r,r+1)} = 
\sum_{r=0}^\infty \left(\sum_{k=0}^r \frac{1}{b(k,k+1)}\right)m(r+1)=\infty.$$
Thus, $L_c$ is essentially self-adjoint by Theorem~\ref{t:bd_characterization}.
\end{proof}

\begin{remark}
For a general birth-death chain, the Feller property of the Neumann semigroup implies
the essential self-adjointness of $L_c$, see \cite{KMW}.
\end{remark}

\subsection{Schr{\"o}dinger operators}
We now discuss the case of Schr{\"o}dinger operators
which act as $\De + W$ where $W \in C(X)$ is a potential
acting on functions via $Wf(x)=W(x)f(x)$.
%\Hm{Is this a general principle?}
First, we note that if we start with an essentially self-adjoint Laplacian, then
a potential that is bounded from below does not
impact the essential self-adjointness.

\begin{proposition}\label{p:bounded_potential}
Let $(\N_0,b,m)$ be a birth-death chain. Let $W \in C(\N_0)$ with $W(x) \geq K$
for all $x \in \N_0$ and some $K \in \R$.
If $L_c$ is essentially self-adjoint, 
then $L_c +W$ is essentially self-adjoint. 
\end{proposition}
\begin{proof}
It suffices to show that if $v \in C(\N_0)$ satisfies $(\Delta + W)v = \lm v$ for $\lm < K$ and $v \in \ell^2(\N_0,m)$,
then $v=0$. If $v(0)=0$, then $v=0$ from $(\Delta + W)v = \lm v$ and induction. 
So, suppose $v(0)>0$. We will 
show that this gives a contradiction to $v \in \ell^2(\N_0,m)$.

By induction, $(\Delta + W)v = \lm v$ is equivalent to
$$v(r+1)-v(r) = \frac{1}{b(r,r+1)} \sum_{k=0}^r (W(k)-\lm) v(k)m(k)$$
see, e.g., \cite{KLW21}. In particular, $v$ is positive and increasing
since $v(0)>0$ and $W(k) - \lm >0$.

Now, let $\lm_1=\sup_{k \in \N_0}( \lm-W(k))<0$
and $w \in C(\N_0)$ satisfy $w(0)=v(0)$ and $\De w = \lm_1 w$. Similar to the above,
it follows by induction that $w$ satisfies
$$w(r+1)-w(r) = \frac{-\lm_1}{b(r,r+1)} \sum_{k=0}^r w(k)m(k)$$
and, thus, $v(r) \geq w(r)$ by a direct comparison
since $-\lm_1 \leq W(k)-\lm$ for all $k \in \N_0$. As $L_c$ is assumed to be essentially self-adjoint,
it follows that $w \not \in \ell^2(\N_0,m)$ and thus $v \not \in \ell^2(\N_0,m)$.
This contradiction implies $v=0$.
\end{proof}

\begin{remark}
We note that, in the above, the form corresponding to $L_c + W$ is semibounded. For general graphs,
whenever such a form is semibounded and the graph is metrically complete with
respect to an intrinsic metric (which implies that $L_c$ is essentially self-adjoint, see Remark~\ref{rem:intrinsic} below), 
then the corresponding operator is essentially self-adjoint, 
see \cite{GKS16} as well as \cite{KMN22} for another proof which involves quantum graphs.
Here, we only assume that $L_c$ is essentially self-adjoint, but make the stronger assumption
of a uniform lower bound on the potential.
\end{remark}

\eat{
We also note that the proof of Proposition~\ref{p:bounded_potential} gives that for any
any Laplacian on a birth-death chain, there exists a potential that makes
the corresponding Schr{\"o}dinger operator essentially self-adjoint. }
We also note that we can always choose a potential that will make
the Schr{\"o}dinger operator essentially self-adjoint. As the result does
not depend on the graph being a birth-death chain, we give a general statement
and proof.
%\Hm{Extended this to all graphs. Is this a general principle as well?}

\begin{proposition}\label{p:potential_help}
Let $(X, b, m)$ be a weighted graph satisfying $\inf_{y \sim x} m(y)>0$ 
for all $x \in X$. 
Then, there exists $W \in C(X)$
such that $L_c +W$ is essentially self-adjoint.
\end{proposition}
\begin{proof}
We will construct $W \geq 0$ such that there is no non-trivial 
$v \in \ell^2(X,m)$ which satisfies $(\Delta + W)v = \lm v$ for $\lm < 0$. 

Suppose that $v(x_0)>0$. From  $(\Delta + W)v(x_0) = \lm v(x_0)$ we
get that there must exist a neighbor $x_1 \sim x_0$ such that
$$ v(x_1) \geq \left( \frac{\Deg(x_0)+W(x_0)-\lambda}{\Deg(x_0)}\right) v(x_0).$$
If not, then $v(y) < \left[ ({\Deg(x_0)+W(x_0)-\lambda})/({\Deg(x_0)})\right] v(x_0)$
for all $y \sim x_0$ which leads to a contradiction to  
$(\Delta + W)v(x_0) = \lm v(x_0)$.
In particular, the above implies $v(x_1) > v(x_0)$.

Iterating this argument, we get a path $(x_k)$ of non-repeating vertices such that
$$ v(x_{k+1}) \geq \left( \frac{\Deg(x_k)+W(x_k)-\lambda}{\Deg(x_k)}\right) v(x_k).$$
Hence, letting $W$ be given by
$W(x) = {\Deg(x)}{\left(\inf_{y \sim x} m(y)\right)}^{-1/2}$
implies
$$\sum_{k=0}^\infty  \left( \frac{\Deg(x_k)+W(x_k)-\lambda}{\Deg(x_k)}\right)^2 m(x_{k+1}) = \infty$$
for every infinite path $(x_k)$.
But this gives that $v$ cannot be in $\ell^2(X,m)$ unless $v=0$. This completes
the proof.
\eat{
If $v(0)=0$, then $v=0$ from $(\Delta + W)v = \lm v$ and induction. So, suppose $v(0)>0$
and by rescaling that $v(0)=1$. 
By induction, $(\Delta + W)v = \lm v$ implies
$$v(r+1)-v(r) = \frac{1}{b(r,r+1)} \sum_{k=0}^r (W(k)-\lm) v(k)m(k) \geq \frac{m(r)}{b(r,r+1)} W(r)$$
which implies
$$v^2(r+1)m(r+1) > \frac{m^2(r)m(r+1)}{b^2(r,r+1)} W^2(r)$$
and, thus, $W$ can be chosen so that $v \not \in \ell^2(\N_0, m)$.}
\end{proof}

\begin{remark}
We note that the technique of the proof was first used in \cite{Woj08}
to establish the essential self-adjointness of the Laplacian on graphs with
counting measure. This technique has subsequently been used to improve
such results in a variety of contexts, see \cite{KL12, BG15, GKS16, Sch20}.
It is not clear to us if the assumption on the measure found in the result
above is essential.
\end{remark}

Next, we present a consequence of a generalized ground state transform which allows us to pass
from a semibounded Schr{\"o}dinger operator to a Laplacian plus a constant. We can then apply
Hamburger's criterion to such operators.
\eat{
Abstractly, this case can be related to the case of a Laplacian by using general transform theory, see
\cite{Tor10, HK11, CTT11, KLW21}. We now give two variants on 
how this idea can be used to give a characterization for the essential self-adjointness of a Schr{\"o}dinger operator.}

\begin{proposition}\label{p:potentials}
Let $(\N_0,b,m)$ be a birth-death chain.
Let $W \in C(\N_0)$ be such that $L_c + W$ is semibounded
as a quadratic form with lower bound $K \in \R$.
Let $v \in C(\N_0)$ with $v>0$ satisfy $(\De + W)v=\lm v$ for $\lm < K$.
Then $L_c + W$ is essentially self-adjoint if and only if 
$$\sum_{r=0}^\infty \left( \sum_{k=0}^r \frac{1}{b(k,k+1)v(k)v(k+1)} \right)^2 v^2(r+1)m(r+1) = \infty.$$
\end{proposition}
\begin{proof}
\eat{
This can be seen by transforming a Schr{\"o}dinger operator to a Laplacian using a generalized eigenfunction.
In this case, the transform is referred to as a ground state transform, see \cite{HK11, KLW21} for more details.
In the construction, we assume that $W \geq K$ for some constant $K$ and 
take a strictly positive function $v >0$ which satisfies $(\De + W)v=\lm v$ for $\lm \leq K$,
see \cite{Dod06, HK11, KLW21} for the construction of $v$. In this case, it follows that the restriction of
$\De +W$ to the finitely supported functions 
is unitarily equivalent to a Laplacian with edge weights $b(x,y)v(x)v(y)$ and vertex measure
$v^2(x)m(x)$ plus a constant potential and thus the essential
self-adjointness can be characterized using Hamburger's criterion.}
Using the isometry
$$T_v \colon \ell^2(\N_0, m v^2) \longrightarrow \ell^2(\N_0,m)$$
defined via $T_vf(x)=v(x)f(x)$ for $x \in X$ and a generalized ground state transform, see \cite{HK11, KPP20, KLW21},
we get that $L_c +W$ on $C_c(\N_0) \subseteq \ell^2(\N_0,m)$ is unitarily equivalent to 
$H$ acting on $C_c(\N_0) \subseteq \ell^2(\N_0, mv^2)$  via
$$H \varphi(x) = \frac{1}{m(x)v^2(x)} \sum_{y \in X} b(x,y)v(x)v(y)(\varphi(x)-\varphi(y)) + \lm \varphi(x)$$
as $T_v^{-1}(L_c+W)T_v=H$.
The result now follows by applying Hamburger's criterion, Theorem~\ref{t:bd_characterization}, to $H$.
\end{proof}

\begin{remark}
The fact that $v>0$ as above exists for $\lm < K $
follows by a generalized Allegretto-Piepenbrink theorem \cite{KPP20}.
\end{remark}

We already know that if $L_c$ is essentially self-adjoint and $W$ is bounded
from below, then $L_c+W$ is essentially self-adjoint.
In principle, the result above could be used to show a type of converse to this,
i.e., that 
if $L_c+W$ is essentially self-adjoint and $W$
does not grow too rapidly, then $L_c$ is essentially self-adjoint. 
However, in order to obtain such a result, we need a better understanding of the behavior
of $v$.
We now give a condition
which implies that $v$ is bounded, thus,
we can remove it from the summability condition.
%The disadvantage of the formulation above is that it involves the generalized eigenfunction $v$. 
%The following gives a case when we can remove this dependence.
\begin{corollary}
Let $(\N_0,b,m)$ be a birth-death chain. Let $W \in C(\N_0)$ with $W \geq 0$
and let $q=Wm$.
If 
$$\sum_{r=0}^\infty \frac{q(B_r) + m(B_r)}{b(r,r+1)} <\infty$$
where $B_r =\{0, 1, 2, \ldots, r\}$, then
$L_c + W$ is essentially self-adjoint if and only if $L_c$ is essentially self-adjoint, i.e.,
$$\sum_{r=0}^\infty \left( \sum_{k=0}^r \frac{1}{b(k,k+1)} \right)^2 m(r+1) = \infty.$$
\end{corollary}
\begin{proof}
By a recursion formula, the solution of
$(\De + W)v=\lm v$ used in Proposition~\ref{p:potentials} is bounded if and only if
$$\sum_{r=0}^\infty \frac{q(B_r) + m(B_r)}{b(r,r+1)} <\infty,$$
see \cite{KLW13, KLW21}. Thus, the result follows
directly from Propositions~\ref{p:bounded_potential}~and~\ref{p:potentials}.
\end{proof}

Finally, we note that if the potential is unbounded below, then the restriction of 
a Laplacian plus a potential may not be essentially self-adjoint even if the restriction
of the Laplacian part is. 
This can be seen in the following, see also Example~5.3.1 in \cite{CTT11}.

\begin{proposition}
Let $(\N_0,b,m)$ be a birth-death chain. Let $\varphi \in C_c(\N_0)$
and let $H_c$ act on $C_c(\N_0)$ as
$$ H_c \varphi(x) = \sum_{y \in X} w(x,y) (\varphi(x)-\varphi(y)) + W(x)\varphi(x)$$
with $w(x,y)=b(x,y)/ \sqrt{m(x)m(y)}$ for $x, y \in \N_0$
and
$$W(x)= \frac{1}{\sqrt{m(x)}}\sum_{y \in X} b(x,y) \left(\frac{1}{\sqrt{m(x)}}-\frac{1}{\sqrt{m(y)}} \right).$$
Then, $H_c$ is essentially self-adjoint on $\ell^2(\N_0,1)$ if and only if 
$$\sum_{r=0}^\infty \left( \sum_{k=0}^r \frac{1}{b(k,k+1)} \right)^2 m(r+1) = \infty.$$
\end{proposition}
\begin{proof}
We follow the ideas of \cite{Tor10, CTT11}.
More specifically,  
letting $v = m^{-1/2}$ and using the isometry $T_v \colon \ell^2(\N_0,1) \longrightarrow \ell^2(\N_0,m)$ given by 
$T_{v}f(x)= f(x)/\sqrt{m(x)},$ it follows that $H_c$ on $C_c(\N_0) \subseteq \ell^2(\N_0,1)$ 
is unitarily equivalent to $L_c$ on $C_c(\N_0) \subseteq \ell^2(\N_0,m)$
as $H_c = T_v^{-1}L_c T_v$. Thus, the result
follows by Hamburger's criterion, Theorem~\ref{t:bd_characterization}.
\end{proof}

We note that the operator $H_c$ in the proposition above
can be thought of as the Laplacian on a birth-death chain 
with edge weights given by $w$ and $m=1$ plus the
potential $W$. 
As in the case of $m=1$, the Laplacian part is always essentially self-adjoint, it follows that
the possible failure of essential self-adjointness of $H_c$
depends solely on the potential term given by $W$.

%On the other hand, starting with a Laplacian which is not essentially self-adjoint,
%it is always possible to choose a potential so that the Laplacian plus the potential
%is essentially self-adjoint as established in Proposition~\ref{p:potential_help}.

%\Hm{Does this follow from above or is this for arbitrary $W$?}
These examples give a contrast
to Schr{\"o}dinger operators which are essentially self-adjoint regardless of the potential.
For example, for a birth-death chain, it can be shown by using stability results and limit point-limit circle theory that
the restriction of $ H f(x) = \sum_{y \in X} w(x,y) (f(x)-f(y)) + W(x)f(x)$ with 
$$w(x,y) = \frac{b(x,y)}{ \sqrt{m(x)m(y)}}$$
to $C_c(\N_0) \subseteq \ell^2(X,m)$ 
will be essentially self-adjoint if
$$\sum_{r=0}^\infty \frac{1}{w(r,r+1)}= \infty$$
regardless of $W$, for details see \cite{Ber68}.

Given the results above, we consider a full characterization of the essential self-adjointness of Schr{\"o}dinger
operators on birth-death chains to be an interesting problem.
We note that one can find an explicit formula for solutions $v$
to the equation $(\Delta + W)v=0$ in the difference equations literature, see
\cite{Mal98};
however, the formula is quite involved and does not seem to yield
a useful characterization for essential self-adjointness.
%\Hm{we should have a look at the Milatovic papers too}

\subsection{Characterization for birth-death chains via capacity}
In this subsection we characterize the essential self-adjointness 
of the Laplacian on birth-death chains via a new notion of capacity.

\eat{
We recall that for a birth-death chain, the energy of a function
is given by
$$\Q(f)= \sum_{r=0}^\infty b(r,r+1) (f(r)-f(r+1))^2$$
and that the
set of functions of finite energy is denoted by
$$\D = \{ f \in C(X) \mid \Q(f) < \infty\}.$$}

We first define the second order Sobolev space $H^2(X,m)$ as
\begin{align*}
H^2(X,m)&= D(L_c^*) \cap \D \\
&= \{f \in \ell^2(X,m) \mid \Delta f \in \ell^2(X,m) \textup{ and } \Q(f)<\infty\}.
\end{align*}
We give the space $H^2(X,m)$ an inner product $\as{\cdot, \cdot}_{H^2}$ by
\[
\as{f, g}_{H^2}=\as{f, g} + \as{L_c^* f, L_c^* g}+ Q(f, g) 
\]
where $\as{\cdot, \cdot}$ denotes the standard inner product on $\ell^2(X,m)$.

\begin{remark}
We note that in \cite{KN23a}, in the context of quantum graphs, the above
space is referred to as the domain of the maximal Gaffney Laplacian.
\end{remark}

We next define the $(2, 2)$-capacity for a birth-death chain. We start with the set of test functions.
For $n \in \N_0$, we let
\begin{equation*}%\label{eq: test function of capacity}
\Cn_n=\{f \in H^2(\N_0,m) \mid f(r) = 1 \text{ for $r \ge n$}\}.
\end{equation*}

We note that any function that is constant from some vertex on has finite energy as the sum appearing
in the definition of the energy form is then finite.
However, the set $\Cn_n$ can be empty as the constant function $1$ may not be in $\ell^2(\N_0,m)$. 
In fact, on a birth-death chain direct arguments yield the following equivalences.
\begin{lemma}\label{l:test_functions}
Let $(\N_0,b,m)$ be a birth-death chain. The following statements are equivalent:
\begin{itemize}
\item[(i)] $\Cn_n \neq \emptyset$ for some $n \in \N_0$.
\item[(ii)] $\Cn_n \neq \emptyset$ for all $ n \in \N_0$.
\item[(iii)] $1 \in H^2(\N_0,m)$.
\item[(iv)] $m(\N_0) < \infty$.
\end{itemize}
\end{lemma}
\eat{
The Cauchy boundary of a graph is usually defined with respect to a metric completion of the graph.
More specifically, given a metric on the set of vertices, we denote by $\ov{X}$ the metric completion
of $X$ and by $\partial_C X = \ov{X} \setminus X$, the \emph{Cauchy boundary} of $X$, see \cite{HKMW13}
for further discussion. However,
for the purposes of our discussion of birth-death chains, it suffices to not specify the metric
and just think of $\partial_C X$ as a point at infinity, i.e., $\partial_C X = \{\infty\}$.}

We next define a notion of capacity for a point at infinity.
\begin{definition}\label{d: capacity}
Let $(\N_0,b,m)$ be a birth-death chain.
The \emph{$(2, 2)$-capacity at infinity} is defined as
\[
\mathrm{Cap}_{2, 2}(\infty)=
\begin{cases}
\infty \quad &\text{if $\Cn_n=\emptyset$ for some $n \in \N_0$} \\
\lim_{n \to \infty} \inf_{f \in \Cn_n} \|f\|_{H^2}^2 &\text{otherwise}.
\end{cases} 
\]
\end{definition}
%\begin{remark}
%In the case of general weighted graph, the definition of capacity should be changed. Namely, we must use the intrinsic metric. 
%\end{remark}

We will approximate the (2,2)-capacity  by the capacity taken over
a smaller set of test functions which vanish up to a certain point. More specifically, 
for $k < n$ we let
%\begin{align*}
%\Cn_{k, n}&=\left\{f \in H^2(\N_0,m)\, \middle|\, 
%\begin{array}{l}
%f(r)=0\, \text{ for $r \le k$} \\
%f(r)=1\, \text{ for $r \ge n$} 
%\end{array}
%\right\} \subseteq \Cn_{n}
%\end{align*}
$$\Cn_{k, n}=\{f \in H^2(\N_0,m) \mid  f(r)=1\, \textup{ for } r \ge n, f(r)=0 \textup{ for } r \le k \}\subseteq \Cn_{n}$$
and
\begin{align*}
\mathrm{Cap}_{2, 2}^{k}(\infty)&=
\begin{cases}
\infty \quad &\text{if $\Cn_{k, n}=\emptyset$ for some $n\in \N_0$} \\
\lim_{n \to \infty}\inf_{f \in \Cn_{k, n}} \|f\|_{H^2}^2 &\text{otherwise}.
\end{cases} 
\end{align*}

We next show that the two notions of capacity introduced above can be compared
via a constant depending only on how far the function is required to be zero.
\begin{lemma}\label{l:key lemma}
Let $(\N_0,b,m)$ be a birth-death chain.
Let $k \in \N$. Then, there exists a constant $C=C(k)>0$ such that
\[
\mathrm{Cap}_{2, 2}(\infty) \le \mathrm{Cap}_{2, 2}^{k}(\infty) \le C \cdot \mathrm{Cap}_{2, 2}(\infty).
\]
\end{lemma}
\begin{proof}
By definition, we obtain $\mathrm{Cap}_{2, 2}(\infty) \le \mathrm{Cap}_{2, 2}^{k}(\infty)$
since $\Cn_{k,n} \subseteq \Cn_n$.
We will prove $\mathrm{Cap}_{2, 2}^{k}(\infty) \le C \cdot \mathrm{Cap}_{2, 2}(\infty)$ for some constant $C>0$
depending only on the choice of $k$ and the graph structure. 

Let $f \in \Cn_{n}$ for $n > k$. We define $\ow{f} \in \Cn_{k, n}$ by
\[
\ow{f}(r)=
\begin{cases}
f(r) \quad & \text{if } r \ge k+1 \\
0 \quad &\text{otherwise}.
\end{cases}
\]
Clearly, $\| \ow{f} \|\le \|f\|$ so that $\ow{f} \in \ell^2(X,m)$. 

Furthermore, 
\begin{align*}
\Q(\ow{f})&=\sum_{r=0}^\infty b(r, r+1)(\ow{f}(r)-\ow{f}(r+1))^2 \\
&=  b(k, k+1)f^2(k+1) + \sum_{r=k+1}^\infty b(r, r+1)(f(r)-f(r+1))^2  \\
&\le \dfrac{b(k,k+1)}{m(k+1)}f^2(k+1) m(k+1) + \Q(f) \\
%+2\dfrac{b(N_1, N_1+1)}{m(N_1)} |u(N_1)|^2m(N_1) + 2\dfrac{b(N_1, N_1+1)}{m(N_2)}  |u(N_2)|^2m(N_2)
&\le C_1 \| f\|_{H^2}^2
\end{align*}
where $C_1 = \max\{1,  \tfrac{b(k, k+1)}{m(k+1)}\}$.
In particular, we obtain $\ow{f} \in \D$.

Finally, direct calculations yield
\begin{equation*}%\label{eq:direct calculation}
%\begin{split}
\Delta \ow{f}(r)=
\begin{cases}
0 \quad & \text{if } r < k \\
-\dfrac{b(k, k+1)}{m(k)}f(k + 1) \quad & \text{if } r=k \\
\Delta f(r) \quad & \text{if } r \ge k+2 \\
\end{cases}
%\end{equation}
%and
%$$(\Delta \ow{u})(N_1+1)&=\dfrac{b(N_1, N_1+1)}{m(N_1+1)}u(N_1+1)  + \dfrac{b(N_1+1, N_1+2)}{m(N_1+1)}(u(N_1+1)-u(N_1+2))$$
%\end{split}
\end{equation*}
and
$$\Delta \ow{f}(k+1)=\frac{b(k, k+1)f(k+1)  + b(k+1, k+2)\big(f(k+1)-f(k+2)\big)}{m(k+1)}.$$ 
Hence,
\begin{align*}
\| \Delta \ow{f} \|^2 &= \frac{b^2(k, k+1)}{m(k)} f^2(k+1)  \\
& \qquad  + \frac{\big| b(k, k+1)f(k+1) + b(k+1, k+2)(f(k+1) - f(k+2)) \big|^2}{m(k+1)} \\
& \qquad + \sum_{r=k+2}^\infty (\Delta f(r))^2m(r) \\
&\leq  \frac{b^2(k, k+1)}{m(k)m(k+1)} f^2(k+1) m(k+1)\\
& \qquad +  \frac{2 b^2(k, k+1)f^2(k+1) + 2 b^2(k+1, k+2)(f(k+1) - f(k+2))^2}{m(k+1)} \\
& \qquad + \| \Delta f \|^2 \\
&\leq \left( \frac{b^2(k, k+1)}{m(k)m(k+1)} + \frac{2 b^2(k, k+1)}{m^2(k+1)} \right) f^2(k+1)m(k+1) \\
& \qquad + \frac{2b(k+1, k+2)}{m(k+1)} \Q(f) + \| \Delta f \|^2 \\
& \leq C_2 \|f\|_{H^2}^2
\end{align*}
where $C_2 = \max \{1, \frac{2b(k+1, k+2)}{m(k+1)},  \frac{b^2(k, k+1)}{m(k)m(k+1)} + \frac{2 b^2(k, k+1)}{m^2(k+1)} \}.$
In particular, this establishes $\ow{f} \in H^2(X,m)$ and, thus, $\ow{f} \in \Cn_{k,n}.$

Hence, for every $f \in \Cn_n$ there exists $\ow{f} \in \Cn_{k,n}$ with
$$\|\ow{f}\|^2_{H^2} \leq C \|f\|^2_{H^2}$$
where $C = C(k)= \max\{ C_1, C_2\}$. 
Taking the infimum over $f \in \Cn_{n}$ and letting $n \to \infty$, we have $\mathrm{Cap}_{2, 2}^{k}(\infty) \le C \cdot \mathrm{Cap}_{2, 2}(\infty)$. This completes the proof.
\eat{
Using the inequality $(\alpha+\beta+\gamma)^2 \le 3(\alpha^2+\beta^2+\gamma^2)$, we obtain
\begin{align*}
|(\Delta \ow{u})(N_1+1)|^2m(N_1+1) &\le 3\dfrac{b(N_1, N_1+1)^2 + b(N_1+1, N_1+2)^2}{m(N_1+1)^2}|u(N_1+1)|^2m(N_1+1)  \\
&\quad + 3 \dfrac{b(N_1+1, N_1+2)^2}{m(N_1+1)m(N_2+1)}|u(N_1+2)|^2m(N_1+2).
\end{align*}
Combining the above inequality and \eqref{eq:direct calculation}, 
\begin{align*}
\|\Delta \ow{u}\|_{\ell^2}^2&=\sum_{x=N_1}^\infty |\Delta \ow{u}(x)|^2\,m(x) \\
&\le \dfrac{b(N_1, N_1+1)^2}{m(N_1)m(N_1+1)} |u(N_1+1)|^2m(N_1+1) \\
&\quad +3\dfrac{b(N_1, N_1+1)^2 + b(N_1+1, N_1+2)^2}{m(N_1+1)^2}|u(N_1+1)|^2m(N_1+1) \\
&\quad +3 \dfrac{b(N_1+1, N_1+2)^2}{m(N_1+1)m(N_1+2)}|u(N_1+2)|^2m(N_1+2) \\
&\quad + \sum_{x=N_1+2}^\infty |(\Delta u)(x)|^2m(x) \\
&\le C(\|u\|_{\ell^2}^2+\|\Delta u\|_{\ell^2}^2)
\end{align*}
where $C=C(N_1, b, m) \ge 1$ is a positive constant. }
\end{proof}

We now have all of the ingredients to prove our characterization of essential self-adjointness in 
terms of capacity.

\begin{theorem}[Characterization via capacity]\label{t:bd_characterization2}
Let $(\N_0,b,m)$ be a birth-death chain. Then, the following statements are equivalent:
\begin{itemize}
\item[(i)] $L_c$ is essentially self-adjoint.
\item[(ii)] $\mathrm{Cap}_{2, 2}(\infty)=0$ or $\mathrm{Cap}_{2, 2}(\infty)=\infty$.
\end{itemize}
\end{theorem}
\begin{proof}
(i) $\Longrightarrow$ (ii): We assume that $L_c$ is essentially self-adjoint.  
If $m(\N_0) = \infty$, then $\mathrm{Cap}_{2, 2}(\infty)=\infty$ 
by Lemma~\ref{l:test_functions} and the definition of capacity. 
Thus, we let $m(\N_0) < \infty$ so that
$1 \in H^2(\N_0,m)$ where $1$ is the constant function which takes the value $1$
on every vertex.

As $L_c$ is essentially self-adjoint, we get $L_c^*=\ov{L}_c$.
As $1 \in H^2(\N_0, m) \subseteq D(L_c^*)=D(\ov{L}_c)$, it follows that
there exists a sequence of functions $\varphi_k \in C_c(X)$ such that $\varphi_k \to 1$
and $L_c^* \varphi_k \to L_c^* 1 =0$ as $k \to \infty$ in $\ell^2(\N_0,m)$.
Putting $f_k=1-\varphi_k$, we note that %this function is eventually constant since 
%$1$ is constant and $\varphi_k \in C_c(X)$. In particular, $f_k$ is 
%in the space of test functions for the capacity, i.e., 
$f_k \in \Cn_{N_k}$ for some $N_k$ and, thus, $f_k \in \Cn_i$ for all $i \geq N_k$. Hence,
we may choose $N_k \to \infty$ and reindex so that $f_k = f_{N_k} \in \Cn_{N_k}$.
Furthermore, $\| f_k \| \to 0$ and
$\| \Delta f_k \| \to 0$ as $k \to \infty$.

By direct calculations and the Cauchy-Schwarz inequality we now obtain the following variant of
Green's formula where all appearing sums are finite
\begin{align*}
\Q(f_k)&=\sum_{r=0}^\infty b(r, r+1)(f_k(r)-f_k(r+1))^2 \\
&=\sum_{r=0}^\infty b(r, r+1)(f_k(r)-f_k(r+1))f_k(r) \\
& \qquad + \sum_{r=1}^\infty b(r, r-1)(f_k(r)-f_k(r-1))f_k(r) \\
%&=\sum_{x=1}^\infty b(x-1, x)(f_k(x)-f_k(x-1))f_k(x)\\
%&\quad - \sum_{x=0}^\infty b(x, x+1)(f_k(x+1)-f_k(x))f_k(x)\\
&=\as{\Delta f_k, f_k} \le \|\Delta f_k\|   \|f_k\|  \to 0 
\end{align*}
as $k \to \infty$.
Hence, as $f_k = f_{N_k} \in \Cn_{N_k}$, we have
\[
\mathrm{Cap}_{2, 2}(\infty) \le \lim_{k \to \infty} \| f_k\|_{H^2}^2=0
\]
which gives (i) $\Longrightarrow$ (ii).

\bigskip 

(ii) $\Longrightarrow$ (i): Assume that $L_c$ is not essentially self-adjoint which, by Theorem~\ref{t:bd_characterization}, 
is equivalent to
\[\sum_{r=0}^\infty \left( \sum_{j=0}^{r} \frac{1}{b(j,j+1)} \right)^2 m(r+1) < \infty.\] 
The above inequality implies $m(\N_0)<\infty$ which gives that $\Cn_n \neq \emptyset $ for all $n \in \N_0$
and $\mathrm{Cap}_{2, 2}(\infty) < \infty$ by Lemma~\ref{l:test_functions}
and the definition of capacity.

Fix $k \in \N_0$. 
We will prove $\mathrm{Cap}_{2, 2}^{k}(\infty)>0$ which implies $\mathrm{Cap}_{2, 2}(\infty)>0$ by Proposition~\ref{l:key lemma}. To prove $\mathrm{Cap}_{2, 2}^{k}(\infty)>0$ it suffices 
to show that there exists a constant $C>0$ such that $C \le \|\Delta f\|$ for all $f \in \Cn_{k, n}$ where $n > k$.

Let $f \in \Cn_{k, n}$ for $n >k$ and let
$$h(r)=\sum_{j=0}^{r-1} \frac{1}{b(j,j+1)}$$
for $r \geq 1$ with $h(0)=0$. 
We note that $h\in \ell^2(\N_0,m)$
by the assumption that $L_c$ is not essentially self-adjoint.
Furthermore, as $f \in \Cn_{k,n}$ is constant everywhere except
for a finite set, 
by direct calculations we obtain the following variant of Green's formula where all appearing sums
are finite
\begin{align*}
&\sum_{r=0}^{\infty}b(r, r+1) (f(r)-f(r+1)) (h(r)-h(r+1)) \\
&=\sum_{r=0}^{\infty}b(r, r+1) (f(r)-f(r+1)) h(r) + \sum_{r=1}^{\infty}b(r,r-1) (f(r)-f(r-1)) h(r) \\
&%=\sum_{j=1}^{\infty}\left( \sum_{|k-j|=1} b(j, k)(u(j)-u(k)) \right)  B(j)
=\as{\Delta f, h}.
\end{align*}
Hence, as $f(0)=0$, $f(n)=1$ and $h(r+1)-h(r) = 1/b(r,r+1)$, we obtain
\begin{align*}
1=f(n)&=\sum_{r=0}^{n-1}(f(r+1)-f(r)) \\
&=\sum_{r=0}^{n-1}b(r, r+1) (f(r+1)-f(r)) (h(r+1)-h(r)) \\
&=\sum_{r=0}^{\infty}b(r, r+1) (f(r)-f(r+1)) (h(r)-h(r+1)) \\
&=\sum_{r=0}^{\infty}\Delta f(r) h(r) m(r).
\end{align*}
By the Cauchy-Schwarz inequality, we now get 
\begin{align*}
1&=\left| \sum_{r=0}^{\infty}\Delta f(r) h(r) m(r)\right| \\
&\le \left( \sum_{r=0}^{\infty} \left( \sum_{j=0}^{r-1} \frac{1}{b(j,j+1)} \right)^2 m(r) \right)^{1/2} \cdot \left( \sum_{r=0}^{\infty}(\Delta f(r))^2m(r)\right)^{1/2} \\
&\le C\|\Delta f\|
\end{align*}
where $C^2=\sum_{r=0}^\infty \left( \sum_{j=0}^{r-1} \frac{1}{b(j,j+1)} \right)^2 m(r)<\infty$. This completes the proof.
\end{proof}

\eat{
From the preceeding considerations, we also get an immediate consequence for the transient case.
\begin{corollary}
Let $(\N_0,b,m)$ be a transient birth-death chain.
Then, 
$$L_c \textup{ is 
essentially self-adjoint } \iff \mathrm{Cap}_{2, 2}(\infty)=\infty.$$
\end{corollary}
\begin{proof}
The statement is immediate from Theorem~\ref{t:bd_characterization2} and
Corollary~\ref{c:transient}. 
\end{proof}}

\subsection{Connections with other criteria}
We now give a series of remarks to put the two characterizations
of essential self-adjointness of the Laplacian on birth-death chains presented
above into context. In particular, we give an overview of recently
appearing results
concerning essential self-adjointness, discuss another capacity
for graphs,
 %,compare the condition on the edge weights and measure with those for recurrence, stochastic completeness and form
%uniqueness 
and then discuss weakly spherically symmetric graphs.

\begin{remark}[Recent results for essential self-adjointness via edge weights, 
vertex measure and intrinsic metrics]\label{rem:intrinsic}
As mentioned already, % the characterization for essential self-adjointness via edge weights
%and vertex measure given in Theorem~\ref{t:bd_characterization2} is known via the connection
%to the moment problem 
Hamburger's test is classically known \cite{Ham20a, Ham20b, Ak65, EK18}. 
Here, we discuss  
some recently appearing results for essential self-adjointness
of the Laplacian as applied to birth-death chains. 

\eat{For the most part, these take the form of sufficient conditions
for essential self-adjointness and are given in terms of conditions
on all infinite paths in a graph. In the case of birth-death chains, this can be reduced to a condition on the 
single infinite path.
The failure of essential self-adjointness is usually established via the failure of form uniqueness which 
we discuss in a subsequent remark.}

We start with \cite{KL12} which establishes that if 
$$\sum_{r=0}^\infty m(r)= \infty,$$ 
then $L_c$ on a birth-death chain is
essentially self-adjoint. This extends a result found for the case of counting measure in \cite{Jor08, JP11, Web10, Woj08}.
In \cite{Gol14}, this is improved to the condition that
$$\sum_{r=0}^\infty \left(\prod_{k=0}^r\left(1 + \frac{1}{\Deg(k)}\right)^2 \right) m(r+1)=\infty,$$
implies that $L_c$ is essentially self-adjoint, see also \cite{GKS16, Sch20}.

A different line of argumentation is pursued in \cite{HKMW13}. There it is shown that if the weighted degree functions is bounded on distance balls
defined with respect to an intrinsic metric, then $L_c$ is essentially self-adjoint.
In particular, this is the case if
all infinite geodesics
have infinite length with respect to an intrinsic path metric. %, then $L_c$ is essentially self-adjoint. 

The condition that all infinite geodesics have infinite length can be understood as a geodesic completeness assumption.
In the setting of birth-death chains, geodesic completeness with respect to an intrinsic path metric means that all symmetric functions
$$\si \colon \N_0 \times \N_0 \longrightarrow [0,\infty)$$ 
with $\si(x,y)>0$
if and only if $|x-y|=1$ and satisfying  
$$b(r,r+1)\si^2(r,r+1)+b(r,r-1)\si^2(r,r-1) \leq m(r),$$ 
must satisfy
$$\sum_{r=0}^\infty \si(r,r+1)=\infty.$$ The quantity $\si(r,r+1)$ is then considered to be the length of the edge connecting
vertices $r$ and $r+1$. For more information about intrinsic metrics and geodesic completeness see
\cite{FLW14, Kel15, KLW21, KM19, Woj21}.

Finally, we mention \cite{Sch20}. There it is shown in the special case of $b(k,k+1)=1$
and $m(r)=r^{-\al}$, that $L_c$ is essentially self-adjoint for $\al \leq 2$, not essentially self-adjoint for $\al>3$
and the case of $2 < \al \leq 3$ is left as an open question. 
From Theorem~\ref{t:bd_characterization}
we note that $L_c$ is essentially self-adjoint for $2 < \al \leq 3$ which completes the picture.
\end{remark}

\eat{
\begin{remark}[Recurrence, stochastic completeness, form uniqueness and the condition in Theorem~\ref{t:bd_characterization}]
We compare the non-summability condition found in Theorem~\ref{t:bd_characterization} to conditions
for other properties on birth-death chains. In particular, we will compare the condition found above to known
characterizations for recurrence, stochastic completeness and form uniqueness for birth-death chains.

We start with recurrence. Here, it is well-known that a birth-death chain is recurrent if and only if
$$\sum_{r=0}^\infty \frac{1}{b(r,r+1)}=\infty,$$
see \cite{Woe09} for a proof. We note that $m$ plays no measure in this criterion.
The second property is stochastic completeness.
Stochastic completeness of birth-death chains is equivalent to
$$\sum_{r=0}^\infty \frac{m(B_r)}{b(r,r+1)} =\infty$$
where $B_r =\{0, 1, 2, \ldots, r\}$, see \cite{KLW13} for a proof. In particular, we note that
recurrence always implies stochastic completeness and that neither property can be related
to essential self-adjointness.

Finally, we mention form uniqueness. As essential self-adjointness deals with the uniqueness of self-adjoint
extensions of $L_c$, form uniqueness deals with the uniqueness of extensions whose corresponding
form is a Dirichlet form. In particular, essential self-adjointness always implies form uniqueness, 
see \cite{HKLW12, KLW21, Sch20} for more details. For form uniqueness on birth-death chains we mention that if $m(\N_0)=\infty$, then $L_c$
is essentially self-adjoint and thus form uniqueness holds. If $m(\N_0)<\infty$, then form uniqueness, recurrence
and stochastic completeness are all equivalent, see \cite{GHKLW15, Sch17} for more details. In particular, for the case when
$m(\N_0)<\infty$, it follows that essential self-adjointness implies all of these properties. Thus, one can use the 
results in the previous remark to give examples where essential self-adjointness fails.

Let us also mention the connection between all of these properties and $\lm$-harmonic functions. As discussed previously,
essential self-adjointness is equivalent to the triviality of $\lm$-harmonic functions which are in $\ell^2(X,m)$ for $\lm<0$.
Similarly, stochastic completeness is equivalent to the triviality of such functions which are bounded instead of in $\ell^2(X,m)$
while form uniqueness to those which are in $\ell^2(X,m)$ and have finite energy, see \cite{HKLW12}. 
For recurrence the situation is a bit different, as recurrence is equivalent to the constancy of positive superharmonic functions,
see \cite{Sch17, Woe00}. The difficulty with using the $\lm$-harmonic characterization for essential self-adjointness
is that $\lm<0$ causes the measure to appear in various places when analyzing generalized eigenfunctions. In the proof
of Theorem~\ref{t:bd_characterization} we circumvent this by reducing to the case of harmonic functions
via our process of doubling and using general theory from \cite{HMW21, Tes00}.
\end{remark}}

\begin{remark}[Capacity and form uniqueness] We mention here a criterion 
from \cite{HKMW13} for form uniqueness of birth-death
chains in terms of a different capacity than the one appearing in Theorem~\ref{t:bd_characterization2}. 
For this, we assume that the birth-death has finite measure, i.e., that $m(\N_0)<\infty$ as, otherwise, the Laplacian is essentially
self-adjoint and, thus, form uniqueness holds.

Now, given a strongly intrinsic path metric and defining the Cauchy boundary as $\partial_C X =
\ov{X} \setminus X$, where the closure is taken with respect to this metric, 
then in \cite{HKMW13} the capacity of the Cauchy boundary is defined similarly
to the above except for the use of functions in the larger set $\D \cap \ell^2(X,m)$
instead of $H^2(X,m) = \D \cap D(L_c^*)$ and removing the Laplacian term. 
In \cite{HKMW13}, it is then shown that
form uniqueness is equivalent to this capacity of the Cauchy boundary being zero.
We note that this result is proven by relying on the Markov property of
$\D \cap \ell^2(X,m)$, i.e., that the space is compatible with normal
contractions. As the Markov property does not hold for $H^2(X,m)$,
the techniques of  \cite{HKMW13} do not apply.
We overcome this difficulty by applying Hamburger's criterion to obtain our capacity result.
\end{remark}

\begin{remark}[Weakly spherically symmetric graphs]
Finally, we mention the case of weakly spherically symmetric graphs. These are graphs that enjoy a certain
symmetry property about a fixed vertex or set, see \cite{KLW13, KLW21, BG15}. In particular, birth-death chains
are the easiest examples of such graphs and thus serve as a toy model. 
Properties such as the Feller property, the $\ell^1$-Liouville property,
recurrence and stochastic completeness have all 
been characterized for such graphs,
see \cite{AS23, Woe09, KLW13, Woj17, KLW21}. \eat{The analysis of properties such as recurrence, stochastic
completeness and form uniqueness discussed above on weakly spherically symmetric graphs can often be reduced
to studying these properties on birth-death chains. The main technique for this reduction is to average generalized eigenfunctions
over spheres. However, in order for this to work, we have to know that the starting eigenfunctions can be taken to be
non-negative. This works for the characterizations of recurrence and stochastic completeness
by considering positive subsolutions, see \cite{HKLW12, Woe00}. However, for essential self-adjointness, we cannot
consider positive subsolutions and thus we are not able to extend the characterization
above to such graphs in such an easy way. }

It would be interesting to extend the characterizations of essential self-adjointness presented in 
Theorems~\ref{t:bd_characterization}~and~\ref{t:bd_characterization2}
to such graphs.
One approach to start to achieve this would be to express the Laplacian on 
certain weakly spherically symmetric
graphs as a direct sum of Jacobi matrices as in \cite{BK13}.
However, in order to use this approach, we would need a workable characterization for
the essential self-adjointness
of Schr{\"o}dinger operators on birth-death chains.
As current formulations are quite complicated, 
a characterization of the essential self-adjointness of $L_c$ on 
weakly spherically symmetric graphs seems to be an interesting open problem.
\end{remark}

\subsection{Necessity of assumptions for stability}
We now use the characterizations above and a general construction to 
give a family of examples that show that some assumptions are needed for a stability result
of the type found in Theorem~\ref{t:stability}.

We first show that if we start with any 
graph and attach a single additional vertex to every vertex,
then there is a choice of edge weights and vertex measure for the new vertices so that the Laplacian on the supergraph is always
essentially self-adjoint. This construction is in the spirit of one found in \cite{KL12} although we simplify the construction a bit
as we only attach a single vertex instead of infinite paths.

\begin{theorem}\label{t:completing}
Let $(X, b, m)$ be a weighted graph satisfying (\ref{A}). Then
there exists an induced supergraph $(\tX, \tb, \tm)$ of $(X,b,m)$
satisfying  (\ref{A})
such that $\widetilde{L}_c$ is essentially self-adjoint.
\end{theorem}
\begin{proof}
Let $X= \{x_k\}_{k=0}^\infty$ denote an enumeration of the vertices of $X$. 
We let $\tX= X \cup \{\ow{x}_k\}_{k=0}^\infty$ and extend $b$
to $\tb$ by letting $\tb\colon \tX \times \tX \longrightarrow [0,\infty)$ by
$\tb \vert_{X \times X} =b$, $\tb(x_k,\ow{x}_k) = \tb(\ow{x}_k, x_k)>0$
and $0$ otherwise. We also extend $m$ to $\tm$ so that $\tm \vert_X = m$.
Thus, $(\tX, \tb, \tm)$ is an induced supergraph of $(X,b,m)$ where
we attach a single additional vertex to every vertex of $X$.

Consider $\ow{\Delta}v = \lm v$ for $\lm<0$ where $\ow{\Delta}$
denotes the formal Laplacian on $(\tX, \tb, \tm)$. We claim
that if 
$$\sum_{k=0}^\infty \left( \frac{\tDeg(\ow{x}_k)}{\tDeg(\ow{x}_k)-\lm}\right)^2 \tm(\ow{x}_k) = \infty,$$
then $v \not \in \ell^2(\tX, \tm)$ unless $v=0$. Here, $\tDeg(\ow{x}_k)= \tb(\ow{x}_k,x_k)/\tm(\ow{x}_k)$
is the weighted degree of $\ow{x}_k$.

If $v(\ow{x}_k)=0$, then from $\ow{\Delta}v(\ow{x}_k) = \lm v(\ow{x}_k)=0$
we get $v(x_k)=0$. Hence, if $v(\ow{x}_k)=0$ for all $k$, then $v=0$.

Now, let $v(\ow{x}_k)>0$ for some $k$. 
From $\ow{\Delta}v(\ow{x}_k) = \lm v(\ow{x}_k)$
we get
$$v(\ow{x}_k) = \left( \frac{\tDeg(\ow{x}_k)}{\tDeg(\ow{x}_k)-\lm}\right) v(x_k).$$
In particular, $v(x_k)> v(\ow{x}_k)>0$. Let $n_1=k$. From
$\ow{\Delta}v(x_{n_1}) = \lm v(x_{n_1}) < 0$
we obtain the existence of a neighbor $x_{n_2} \in X$ of $x_{n_1}$ such that
$v(x_{n_2}) > v(x_{n_1})$. From $\ow{\Delta}v(\ow{x}_{n_2}) = \lm v(\ow{x}_{n_2})$ we now obtain
$$v(\ow{x}_{n_2}) = \left( \frac{\tDeg(\ow{x}_{n_2})}{\tDeg(\ow{x}_{n_2})-\lm}\right) v(x_{n_2}) >  \left( \frac{\tDeg(\ow{x}_{n_2})}{\tDeg(\ow{x}_{n_2})-\lm}\right) v(x_{n_1}).$$
Iterating the argument, we get a path $(x_{n_k})$ such that
$v(x_{n_{k-1}}) < v(x_{n_k})$ and
$$v(\ow{x}_{n_k}) >  \left( \frac{\tDeg(\ow{x}_{n_k})}{\tDeg(\ow{x}_{n_k})-\lm}\right) v(x_{n_1})$$
for all $k$.
Thus, by the assumption on $\tDeg$ and $\tm$, 
$v \not \in \ell^2(\tX, \tm)$ unless $v =0$. 
This completes the proof.
\end{proof}

\begin{remark}
To satisfy the condition above it suffices, for example, that $\ow{b}(x_k, \ow{x}_k)=1 =\ow{m}(\ow{x}_k)$.
\end{remark}

We now apply the above result to a birth-death chain with a single vertex
attached to each natural number. Pictorially the graph in the next result 
can be represented as in Figure~\ref{f:counter_ex}.

\begin{corollary}\label{c:bd_stability_break}
Let $(\N_0,b,m)$ be a birth-death chain.
Let $\tX=\N_0 \cup \{x_k\}_{k=0}^\infty$ and let
$(\tX, \tb, \tm)$ be an induced supergraph of the birth-death chain
by letting $\tb\vert_{X \times X}=b$,  
$\tb(k,x_k)=\tb(x_k,k)>0$ and $b$ zero otherwise
and by extending $m$ to $\tm$.  
If
$$\sum_{r=0}^\infty \left(\frac{\tDeg(x_r)}{\tDeg(x_r)+1}\right)^2 \tm(x_r) = \infty \quad \textup{ and } \quad
\sum_{r=0}^\infty \left( \sum_{k=0}^r \frac{1}{b(k,k+1)} \right)^2 m(r+1) < \infty, $$
then $\ow{L}_c$ is essentially self-adjoint while $L_{c}$ is not
essentially self-adjoint.
\end{corollary}

\begin{proof}
The result follows by applying the proof of Theorem~\ref{t:completing}
which shows that $\ow{L}_c$ is essentially self-adjoint under
the assumptions on $\tDeg$ and $\tm$
and Theorem~\ref{t:bd_characterization} which implies that $L_c$ is not
essentially self-adjoint.
\eat{
In order to establish the essential self-adjointness of $L_c$, we analyze solutions $v \in \ms{F}$ to $\De v = \lm v$
for $\lm<0$. If $v(x_0)=0,$ then $\De v = \lm v$ implies $v=0$. Thus, we assume that $v(x_0)>0$
and wish to show that $v \not \in \ell^2(X,m)$. 

A direct calculation using $\De v(x_0) =\lm v(x_0)$
gives
$$0<v(x_0) =\left(\frac{\Deg(x_0)}{\Deg(x_0)-\lm}\right) v(0) < v(0)$$
where $\Deg$ is the weighted vertex degree. Furthermore, using $\De v(0) = \lm v(0)<0$, now gives
$v(0)<v(1)$. Analogous calculations and induction  imply
$$v(x_r) =\left(\frac{\Deg(x_r)}{\Deg(x_r)-\lm}\right) v(r)$$
and $v(r) < v(r+1)$ for all $r \in \N_0$. In particular, if 
$$\sum_{r=0}^\infty \left(\frac{\Deg(x_r)}{\Deg(x_r)-\lm}\right)^2 m(x_r) = \infty,$$
then $v \not \in \ell^2(X,m)$ and thus $L_c$ is essentially self-adjoint.
On the other hand, from Theorem~\ref{t:bd_characterization}, we know that the essential self-adjointness of $L_{c,1}$ is equivalent
to $\sum_{r=0}^\infty \left( \sum_{k=0}^r 1/b(k,k+1) \right)^2 m(r+1) = \infty.$ 
This completes the proof.}
\end{proof}

\begin{figure}[htbp]
\centering
\includegraphics[width=10cm]{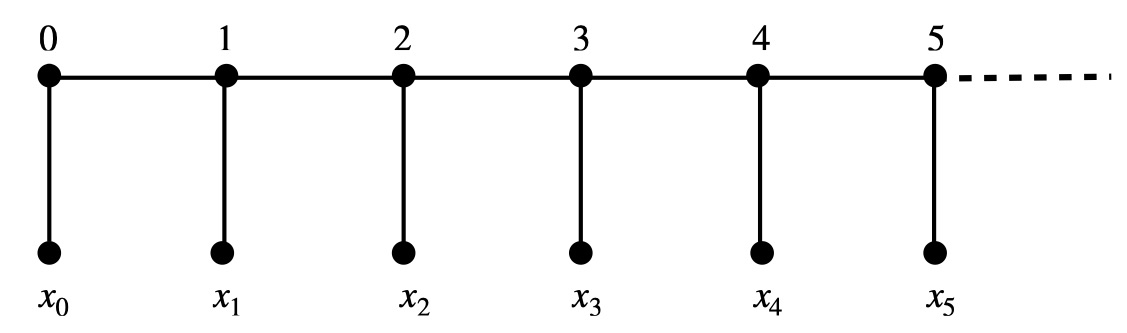}
\caption{The graph in Corollary~\ref{c:bd_stability_break}}
\label{f:counter_ex}
\end{figure}

\begin{remark}
We recall that, in the statement of Theorem~\ref{t:stability}, 
$X=X_1 \cup X_2$ is a disjoint union
with $X_3 = \pt X_1 \cup \pt X_2$ and $b_\partial(x, y)= b(x, y)$ if $x \in \partial X_1, y \in X_2$ or $x \in \partial X_2, y \in X_1$ and 0 otherwise. 
In the corollary above, if we consider 
$X_1= \N_0$, $X_2=\{x_k\}_{k=0}^\infty$, then 
$$\tDeg_\partial(x_r)= \tDeg(x_r) = \frac{\tb(r, x_r)}{\tm(x_r)} \qquad \textup{ and }
\qquad \tDeg_\partial(r) = \frac{\tb(r,x_r)}{m(r)}.$$ 
Under the assumptions in the corollary above, 
both of these cannot be bounded by Theorem~\ref{t:stability}.
This can also be checked directly.% \Hm{Should we add this?}

Furthermore, we note that Theorem~\ref{t:completing} complements
the result from Proposition~\ref{p:potential_help}. In 
Proposition~\ref{p:potential_help}, any Laplacian
can be made essentially self-adjoint by adding a positive potential whereas,
in Theorem~\ref{t:completing}, the essential self-adjointness is achieved by geometric means.
\end{remark}
\eat{
We denote the corresponding
weighted degree function for $b_\partial$ by $\Deg_\partial$. We also recall that we restrict $b$ to $X_k \times X_k$
and $m$ to $X_k$ to get $b_k$ and $m_k$ for $k=1,2$. We denote the restriction of the corresponding Laplacian by $L_{c,k}$.}

\eat{
\begin{example}\label{ex:counter}
We let $X=\N_0 \cup \{x_k\}_{k=0}^\infty$ and let $b(x,y)>0$ for $x,y \in \N_0$ if and only if $|x-y|=1$
and $b(k,x_k)=b(x_k,k)>0$ with $b$ zero otherwise. Thus, pictorially, the graph can be represented as in Figure~\ref{f:counter_ex}.

\begin{figure}[htbp]
\centering
\includegraphics[width=10cm]{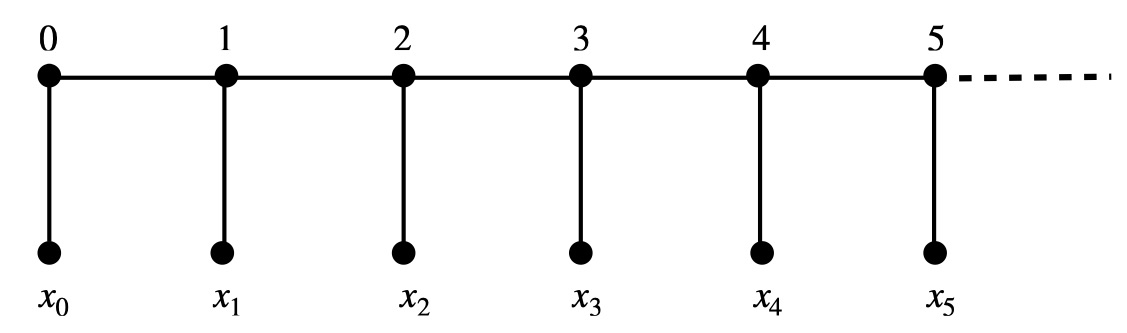}
\caption{}
\label{f:counter_ex}
\end{figure}

We let $X_1=\N_0$ and $X_2 = \bigcup_{k=0}^\infty \{x_k\}$ so that $X_3 = \pt X_1 \cup \pt X_2= X$ with $b_\partial$ giving the vertical edges
connecting each $k \in \N_0$ to the vertex $x_k$.
We now show that by, choosing appropriate edge weights and measure, the Laplacian on this graph is always essentially self-adjoint.
We note that $L_{c,1}$ may or may not be essentially self-adjoint as we have seen in Theorems~\ref{t:bd_characterization}~and~\ref{t:bd_characterization2}.
On the other hand, as $L_{c,2}$ is a Laplacian on a graph with no edges, it is always essentially self-adjoint.

In order to give a criterion for essential self-adjointness of $L_c$ we analyze solutions $v \in \ms{F}$ to $\De v = \lm v$
for $\lm<0$. If $v(x_0)=0,$ then $\De v = \lm v$ implies $v=0$. Thus, we assume that $v(x_0)>0$
and wish to give a condition so that $v \not \in \ell^2(X,m)$. A direct calculation using $\De v(x_0) =\lm v(x_0)$
gives
$$0<v(x_0) =\left(\frac{\Deg(x_0)}{\Deg(x_0)-\lm}\right) v(0) < v(0)$$
where $\Deg$ is the weighted vertex degree. Furthermore, using $\De v(0) = \lm v(0)<0$, now gives
$v(0)<v(1)$. Analogous calculations and induction  imply
$$v(x_r) =\left(\frac{\Deg(x_r)}{\Deg(x_r)-\lm}\right) v(r)$$
and $v(r) < v(r+1)$ for all $r \in \N_0$. In particular, if 
$$\sum_{r=0}^\infty \left(\frac{\Deg(x_r)}{\Deg(x_r)-\lm}\right)^2 m(x_r) = \infty,$$
then $v \not \in \ell^2(X,m)$ and thus $L_c$ is essentially self-adjoint.

Therefore, choosing $b$ and $m$ so that
$$\sum_{r=0}^\infty \left(\frac{\Deg(x_r)}{\Deg(x_r)-\lm}\right)^2 m(x_r) = \infty \quad \textup{ and } \quad
\sum_{r=0}^\infty \left( \sum_{k=0}^r \frac{1}{b(k,k+1)} \right)^2 m(r+1) < \infty $$
gives a graph such that $L_c$ is essentially self-adjoint but $L_{c,1}$ is not
essentially self-adjoint.

\eat{

Given that $\lim_{r \to \infty} m(x_r)/b(r, x_r)$ exists and is finite, and
using the limit comparison test for series and the fact that $\Deg(x_r) = b(r,x_r)/m(x_r)$,
we get that the above series diverges if and only if
$$\sum_{r=0}^\infty \frac{b^2(r,x_r)}{m(x_r)} = \infty.$$
This gives a sufficient condition for the essential self-adjointness of $L_c$.

From Theorem~\ref{t:bd_characterization}, we know that the essential self-adjointness of $L_{c,1}$ is equivalent
to $\sum_{r=0}^\infty \left( \sum_{k=0}^r 1/b(k,k+1) \right)^2 m(r+1) = \infty.$ 
Thus, we see that by choosing $b$ and $m$ so that $\lim_{r \to \infty} m(x_r)/b(r, x_r)$ exists and is finite,
$$\sum_{r=0}^\infty \left( \sum_{k=0}^r \frac{1}{b(k,k+1)} \right)^2 m(r+1) < \infty \qquad \textup{ and } \qquad
\sum_{r=0}^\infty \frac{b^2(r,x_r)}{m(x_r)} = \infty $$
we get that $L_{c,1}$ is not essentially self-adjoint while $L_c$ is essentially self-adjoint. }

We note that for this example 
$$\Deg_\partial(x_r)= \Deg(x_r) = \frac{b(r, x_r)}{m(x_r)} \qquad \textup{ and }
\qquad \Deg_\partial(r) = \frac{b(r,x_r)}{m(r)}.$$ 
Under the assumptions above, both of these cannot be bounded by Theorem~\ref{t:stability}.
This can also be checked directly. \Hm{should we add this?}
\end{example}}

\subsection{Characterizations for star-like graphs}
In this subsection we extend the notion of star-like graphs from \cite{CTT11} and
give a characterization for the Laplacian to be essentially self-adjoint on such graphs.
Star-like graphs are graphs such that when a subset of vertices is removed
what is left consists of a disjoint union birth-death chains. 
In \cite{CTT11}, this idea is applied to the removal
of a finite subset. Here, we relax this assumption considerably and modify
it to better fit the framework for studying essential self-adjointness.
We again use the notations established for the statement of Theorem~\ref{t:stability}, i.e.,
$X=X_1 \cup X_2$ is a disjoint union with 
 with $(X_i,b_i,m_i)$ being the induced subraphs for $i=1,2$,
$X_3 = \pt X_1 \cup \pt X_2$ and $\Deg_\partial$
giving the vertex degree connecting the two subgraphs.

\begin{definition}[Star-like graph]
A weighted graph $(X,b,m)$ satisfying (A) is called \emph{star-like (for essential self-adjointness)} if $X_1 \subseteq X$ can be chosen so that
\begin{enumerate} 
\item $L_{c,1}$ is essentially self-adjoint 
\item $\Deg_\partial$ is bounded
\item $(X_2,b_2,m_2)$ is a disjoint union of birth-death chains.
\end{enumerate}
A birth-death chain in $X_2$ is called a \emph{ray} of the graph provided that $X_1$
can be chosen so that there is exactly one edge connecting the birth-death chain to $X_1$. 
When $X_1$ consists of a single vertex and all birth-death chains are rays, we call the graph a \emph{star}.
\end{definition}

We now give an example to show that star-like graphs in our sense may have no rays. 
This graph may be visualized as an infinite wheel.
\begin{example}[Star-like but no rays]
Let $X=\{x_0, 0, 1, 2, \ldots\}$ with $X_1=\{x_0\}$, $X_2=\N_0$
and choose $b$ so that $b(j,k)>0$ if $ j,k  \in \N_0$ and $|j-k|=1$,
$b(x_0, k)>0$ for all $k \in \N_0$ with $\sum_{k \in \N_0}b(x_0,k)<\infty$ and $0$ otherwise.
We further choose $m$ so that (A) is satisfied
and $\Deg_\partial(k)= b(k,x_0)/m(k)$ is bounded for $k \in \N_0$
so that the graph is star-like for essential self-adjointness.
Then, $X_2$ consists of a single birth-death chain which is not a ray as there are infinitely many
edges connecting $X_2$ to $X_1$.
\end{example}

\begin{remark}
We note that for $L_{c,1}$ to be essentially self-adjoint it suffices, for example, for any one of the following
conditions to hold:
\begin{itemize}
\item $\Deg_1$ is bounded on $X_1$.
\item The measure of each infinite path in $X_1$ is infinite \cite{KL12}. 
\item The graph $(X_1,b_1,m_1)$ allows for an intrinsic metric
such that all distance balls defined with respect to this metric are finite
\cite{HKMW13}. 
\end{itemize}
In particular, all of these assumptions hold if $X_1$ is a finite set as in \cite{CTT11}.
\end{remark}

We write
$$X_2 = \bigsqcup_{i \in I} \N_0^{(i)}$$
with $\N_0^{(i)}=\{k^{(i)}\}_{k=0}^\infty$
%We note that both the set $X_1$ and the number of birth-death chains in $X_2$ 
%can be infinite.
%In particular, as we do not assume local finiteness, the star with infinitely
%many rays falls into our framework provided that
%$\Deg_\partial$ is bounded, see Example~\ref{ex:esa_star} below.
and $L_c^{(i)}$ for the formal Laplacian arising from 
$ (\N_0^{(i)}, b,m)$ restricted to the finitely supported
functions.
Disjointness here means that $b(k^{(i)}, l^{(j)}) =0$ whenever
$i \neq j$ for all $k, l \in \N_0$.
Furthermore, we let $\infty^{(i)}$ denote a point at infinity for the birth-death chain
$(\N_0^{(i)}, b,m)$ 
 and let $\mathrm{Cap}_{2, 2}(\infty^{(i)})$ denote the 
$(2,2)$-capacity of this point. 

With these notions, we
then directly have the following characterization of essential self-adjointness for such graphs.
\begin{theorem}[Characterization for star-like graphs]\label{t:star-like_characterization}
Let $(X,b,m)$ be a star-like graph for essential self-adjointness. Then, the following statements are equivalent:
\begin{itemize}
 \item[(i)] $L_c$ is essentially self-adjoint.
 \item[(ii)] $L_c^{(i)}$ is essentially self-adjoint for every $i \in I$.
\item[(iii)] $\displaystyle{\sum_{r=0}^\infty \left( \sum_{k=0}^r \frac{1}{b\left(k^{(i)},(k+1)^{(i)}\right)} \right)^2 m\left((r+1)^{(i)}\right) = \infty}$
for every $i\in I$.
\item[(iv)] $\mathrm{Cap}_{2, 2}(\infty^{(i)})=0$ or $\mathrm{Cap}_{2, 2}(\infty^{(i)})=\infty$
for every $i \in I$.
\end{itemize}
\end{theorem}
\begin{proof}
This follows immediately from Theorems~\ref{t:stability},~\ref{t:bd_characterization},~\ref{t:bd_characterization2}
and general theory. In particular, as the Laplacian $L_{c,1}$ is assumed to be essentially self-adjoint
and $\Deg_\partial$ is bounded, Theorem~\ref{t:stability} gives that the essential self-adjointness of 
$L_c$ is equivalent to the essential self-adjointness
of the Laplacian $L_{c,2}$. By the disjointness of the birth-death chains, 
the Laplacian $L_{c,2}$ can be decomposed into a direct sum of
$L_c^{(i)}$, 
thus, $L_{c,2}$ is essentially self-adjoint if and only if 
every $L_c^{(i)}$ is essentially self-adjoint which gives the equivalence
between (i) and (ii). The equivalence between (ii), (iii), and (iv) follows from 
Theorems~\ref{t:bd_characterization}~and~\ref{t:bd_characterization2}.
\end{proof}

We now give two examples to illustrate the result.

\begin{example}[Star with two rays]\label{ex:esa_integers}
In the case of $X=\Z$ with $b(x,y)>0$ if and only if $|x-y|=1$ we get that the Laplacian $L_c$ is essentially self-adjoint 
if and only if
$$\sum_{r=0}^\infty \left( \sum_{k=0}^r \frac{1}{b(k,k+1)} \right)^2 m(r+1) = \infty$$
and
$$\sum_{r=0}^\infty \left( \sum_{k=0}^r \frac{1}{b(-k,-k-1)} \right)^2 m(-r-1) = \infty$$
by taking $X_1=\{0\}$ in Theorem~\ref{t:star-like_characterization}.
\end{example}

\begin{example}[Star with infinitely many rays]\label{ex:esa_star}
For this example, we let $X_1=\{x_0\}$ and 
$$X_2 = \bigsqcup_{i \in \N} \N_0^{(i)}$$ 
with $b(x_0, 0^{(i)})>0$ for $i \in \N$ 
satisfying 
$$\sum_{i =1}^\infty b(x_0, 0^{(i)})<\infty \qquad \textup{ and } \qquad \sum_{i=1}^\infty \frac{b^2(x_0,0^{(i)})}{m(0^{(i)})}< \infty.$$
In order to apply Theorem~\ref{t:star-like_characterization}, we need that $\Deg_\partial$ is bounded. %\Hm{can this be removed by Kato-Rellich?}
In this case, this amounts to the fact that the function $f\colon \N \longrightarrow [0,\infty)$ given by
$$f(i)= \frac{b(0^{(i)},x_0)}{m(0^{(i)})}$$ 
is bounded.

Under these assumptions, Theorem~\ref{t:star-like_characterization} gives that the Laplacian on the infinite
star is essentially self-adjoint if and only if the Laplacian on each birth-death chain is essentially self-adjoint, i.e.,
$$\sum_{r=0}^\infty \left( \sum_{k=0}^r \frac{1}{b\left(k^{(i)},(k+1)^{(i)}\right)} \right)^2 m\left((r+1)^{(i)}\right) = \infty$$
for all $i\in \N$.

%We note that the infinite star has already shown to be a rich source of various counterexamples, see \cite{HK11, KS20}
%for some examples.
\end{example}

\subsection{Characterizations of the $\ell^2$-Liouville property}
We have seen that the analysis of harmonic functions plays a key role in the preceding considerations.
In particular, in order to establish the failure of essential self-adjointness on birth-death chains, 
we constructed a non-constant harmonic function in $\ell^2$ over the doubled birth-death chain.
In this subsection, we focus a bit more on the existence of non-constant harmonic
functions which are square summable.

In particular, we recall that a weighted graph satisfies the $\ell^2$-\emph{Liouville property} if every harmonic
function which is in $\ell^2(X,m)$ is constant. A general
criterion for this to happen is a Karp-type theorem which states that this is the
case if the harmonic function does not grow too rapidly on balls defined with respect
to an intrinsic metric. A consequence
of this is a Yau-type theorem which implies the $\ell^2$-Liouville
property in case of metric completeness with respect to an intrinsic metric,
see \cite{HK14, KLW21} for further details as well as \cite{HKLS22} for
an extension to general Dirichlet forms.

The essential self-adjointness of $L_c$ always implies 
the $\ell^2$-Liouville property. On the other hand, the $\ell^2$-Liouville property along with strict positivity of the
Laplacian
and $m(X)=\infty$, implies the essential self-adjointness of $L_c$. This follows from general theory, see \cite{GM11, HMW21}.

%We note that every birth-death chains satisfies the $\ell^2$-Liouville property as every harmonic
%function on a birth-death chain is constant by induction. Thus, the $\ell^2$-Liouville
%property alone does not imply essential self-adjointness for birth-death chains. 

We will now focus on the $\ell^2$-Liouville property in the case of 
star-like graphs as introduced in the previous subsection and give some 
further connections to essential self-adjointness.
We first extend Lemma~\ref{l:harmonic} concerning 
harmonic functions on birth-death chains to stars with two rays, i.e., graphs with $X=\Z$
and $b(x,y)>0$ if and only if $|x-y|=1$.

\begin{lemma}[Characterization
of harmonic functions on stars with two rays]\label{l:double_harmonic}
Let $(\Z,b,m)$ be a star with two rays. Let $v \in C(\Z)$ and $C=b(0,1)(v(1)-v(0))$.
Then, $v$ is harmonic
if and only if all three of the following conditions hold:

\begin{enumerate}
\item If $r \geq 1$, then 
$$v(r+1) = v(1) + C \sum_{k=1}^r \frac{1}{b(k,k+1)}.$$
\item $C=b(0,-1)(v(0)-v(-1))$.
\item If $r \geq 1$, then
$$v(-r-1) = v(-1) - C \sum_{k=1}^r \frac{1}{b(-k,-k-1)}.$$ 
\end{enumerate}
In particular, $v$ is constant if and only if $C=0$.
\end{lemma} 
\begin{proof}
The proof follows by direct calculations and induction.
\end{proof}

As a direct consequence,
we first note that there is a special case when the harmonic
function vanishes at either end of the graph.
\begin{corollary}\label{c:vanishing}
Let $(\Z,b,m)$ be a star with two rays. Let 
$v \in C(\Z)$ be harmonic and let $C=b(0,1)(v(1)-v(0))$.
Then, $v(r) \to 0$ as $r \to \infty$ if and only if 
$\sum_{k=0}^\infty b^{-1}(k,k+1)<\infty$ and
$$C={-v(1)}/{\sum_{k=0}^\infty b^{-1}(k,k+1)<\infty}.$$
In this case, 
$$v(r)=-C\sum_{k=r}^\infty \frac{1}{b(k,k+1)}$$
for $r \geq 2$. 

Analogously,
 $v(r) \to 0$ as $r \to -\infty$ if and only if 
$\sum_{k=0}^\infty b^{-1}(-k,-k-1)<\infty$ and
$$C={v(-1)}/{\sum_{k=0}^\infty b^{-1}(-k,-k-1)<\infty}.$$
In this case, 
$$v(-r)=C\sum_{k=r}^\infty \frac{1}{b(-k,-k-1)}$$
for $r \geq 2$.  
\end{corollary}
\begin{proof}
The proof follows directly from Lemma~\ref{l:double_harmonic}
and basic algebraic manipulations.
\end{proof}

We remark that when $v$ vanishes at either end, it agrees 
with the
Green's function on that end. We will return to this viewpoint later
when we characterize the $\ell^2$-Liouville property.

As a second direct consequence, we merely reiterate when such a harmonic function
is in $\ell^2$ for future reference.
\begin{corollary}\label{c:v_square_sum}
Let $(\Z,b,m)$ be a star with two rays. 
Let $v \in C(\Z)$ be harmonic with $C=b(0,1)(v(1)-v(0))$. 
Then, $v \in \ell^2(\Z,m)$ if and only if
$$\sum_{r = 1}^\infty\left(v(1) + C \sum_{k=1}^r \frac{1}{b(k,k+1)}\right)^2
m(r+1)<\infty$$
and
$$\sum_{r = 1}^\infty\left(v(-1) - C \sum_{k=1}^r \frac{1}{b(-k,-k-1)}\right)^2
m(-r-1)<\infty.$$
\end{corollary}

\eat{
As another direct consequence of the above, we give 
conditions on some specific harmonic functions to be constant.
\begin{corollary}\label{c:dbd_constant}
Let $(\Z,b,m)$ be a star with two rays. 
Let $v \in C(\Z)$ be harmonic such that $v \in \ell^2(\Z,m)$
and let $C=b(0,1)(v(1)-v(0)).$
\begin{enumerate}
\item If $Cv(1)\geq0$ and 
$$\sum_{r=1}^\infty \left( \sum_{k=1}^r \frac{1}{b(k,k+1)}\right)^2 m(r+1)
=\infty,$$
then $v$ is constant.
\item If $Cv(1)\leq 0$ and
$$\sum_{r=1}^\infty \left( \sum_{k=1}^r \frac{1}{b(-k,-k-1)}\right)^2 m(-r-1)
=\infty,$$
then $v$ is constant.
\end{enumerate} 
\end{corollary}
\begin{proof}
For the proof of $(1)$ we note that if $Cv(1)\geq0$, then either
$C \geq 0$ and $v(1) \geq 0$ or $C \leq 0$ and $v(1) \leq 0$. In either case,
if $ C \neq 0$, then 
$$\sum_{r=1}^\infty \left( \sum_{k=1}^r \frac{1}{b(k,k+1)}\right)^2 m(r+1)
=\infty$$
is equivalent to
$$\sum_{r = 1}^\infty\left(v(1) + C \sum_{k=1}^r \frac{1}{b(k,k+1)}\right)^2
m(r+1)=\infty$$
by basic arguments. This gives a contradiction to $v \in \ell^2(\Z,m)$
by Corollary~\ref{c:v_square_sum}. Therefore, $C=0$ and thus
$v$ must be constant by Lemma~\ref{l:double_harmonic}.

To prove $(2)$, we note that $Cv(1) \leq 0$ is equivalent to $Cv(-1) \geq 0$.
Now, the proof follows as the proof of $(1)$ above.
\end{proof}

We note that the non-summability conditions above are exactly those
found for the essential self-adjointness of $L_c^{(1)}$ and
$L_c^{(2)}$ where we take $X_1=\{0\}$ and let
$L_c^{(1)}=L_c^{(\infty)}$ and $L_c^{(2)}=L_c^{(-\infty)}$ be the Laplacians on the positive
and negative rays, respectively.
From Theorem~\ref{t:stability} and general theory, we know that $L_c$
is essentially self-adjoint if and only if both $L_c^{(1)}$ and
$L_c^{(2)}$ are essentially self-adjoint which then implies
that the star with two rays satisfies 
the $\ell^2$-Liouville property. This also follows directly from Corollary~\ref{c:dbd_constant} and Theorem~\ref{t:star-like_characterization}.}

\begin{remark}[Harmonic functions on stars with two rays]
Looking at the results above, we see that non-constant 
harmonic functions on stars with two rays are always monotonic, 
thus, there are only a few possibilities for their end behavior. 
Either they remain bounded at either or both ends, in which case it is easy to characterize when they are
in $\ell^2$ there depending 
on if they go to zero or to a non-zero constant on that end. 
In the case that they are unbounded on an end,
the function being in $\ell^2$ there is equivalent to 
the failure of the essential self-adjointness
of the Laplacian by the limit comparison test and Hamburger's
criterion. These observations will ultimately lead to our characterization.
\end{remark}

We start with the following which states that the
$\ell^2$-Liouville property implies that the Laplacian on at 
least one of the rays of a star  is essentially self-adjoint.
\begin{proposition}
Let $(X,b,m)$ be a star with at least two rays. 
If $(X,b,m)$ satisfies the $\ell^2$-Liouville property,
then $L_c^{(i)}$ is essentially self-adjoint for some $i \in I$.
\end{proposition}
\begin{proof}
We denote the common vertex of the rays by $x_0$.
Suppose that $L_c^{(i)}$ are not essentially self-adjoint for all $i$.
Since there are at least two rays this means that there exist $L_c^{(i_1)}$
and $L_c^{(i_2)}$ which are not essentially self-adjoint.
Then we can define a non-constant harmonic function by letting $v(x_0)=0, v(0^{(i_1)})>0$ 
and using Lemma~\ref{l:double_harmonic}
to extend $v$ to $\N_0^{(i_1)}$ and $-\N_0^{(i_2)}$.
We then let $v$ be 0 at all other vertices. By direct calculations
and Lemma~\ref{l:double_harmonic}, this function is harmonic.
Furthermore, $v \in \ell^2(X,m)$ by Theorem~\ref{t:star-like_characterization},
Corollary~\ref{c:v_square_sum} and 
basic arguments.
\end{proof}

\begin{remark}
We need at least two rays in the result above as all birth-death chains
satisfy the $\ell^2$-Liouville property but $L_c$ may not be essentially self-adjoint.
\end{remark}

The above result reverses in the case of recurrent stars with exactly two rays.
This begins our characterization of the $\ell^2$-Liouville property in this
case. For the sake of consistency, we formulate our results
in terms of the failure of this property as further proofs will take this 
viewpoint and be constructive
in nature.
Also for the sake of consistency with future statements, we let
$L_c^{(\infty)}$ and $L_c^{(-\infty)}$ denote the Laplacians on the
negative and positive parts of the graph. 
As all harmonic functions are unbounded at both ends in the recurrent case,
they may be visualized as in Figure~\ref{fig:M1}.
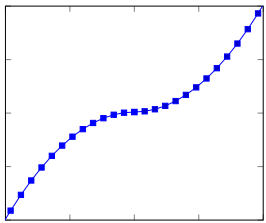
\begin{figure} \centering
\begin{tikzpicture}[scale=0.50]
\begin{axis}
[
xlabel={},
ylabel={},
xmin=-10, xmax=10,
ymin=-100, ymax=100,
 yticklabel=\empty,
  xticklabel=\empty,
]
\addplot+[y,mark=square*, domain=1:10, samples at={0,0.8,..., 10},circle, blue]{x*x+1};
\addplot+[y, mark=square*,,mark options={solid,fill=blue}, domain=1:10, samples at={0,-0.8,...,-10}, circle, blue]{-x*x+1};
\end{axis}
\end{tikzpicture}
\caption{Harmonic functions: recurrence} \label{fig:M1}
\end{figure}

\begin{proposition}[Characterization of $\ell^2$-Liouville: recurrence]\label{p:recurrent_rays}
Let $(\Z,b,m)$ be a star with two rays which is recurrent. 
Then, $(\Z,b,m)$ does not satisfy the $\ell^2$-Liouville property
if and only if both $L_c^{(\infty)}$ and $L_c^{(-\infty)}$ are not essentially self-adjoint.
\end{proposition}
\begin{proof}
One implication is contained in the above and does not require recurrence. For the other implication,
we suppose that $L_c^{(\infty)}$ is essentially self-adjoint. 
Let $v$ be a harmonic function. As recurrence is equivalent
to $\sum_{k=0}^\infty 1/b(k,k+1)= \infty = \sum_{k=0}^\infty 1/b(-k,-k-1)$ 
it follows
that if $C \neq 0$, then the fact that $L_c^{(\infty)}$
is essentially self-adjoint is equivalent to 
$$\sum_{r = 1}^\infty\left(v(1) + C \sum_{k=1}^r \frac{1}{b(k,k+1)}\right)^2
m(r+1)=\infty$$
by Theorem~\ref{t:star-like_characterization} and the limit comparison test.
This would imply that $v$ is not in $\ell^2(\Z,m)$ by Corollary~\ref{c:v_square_sum} 
which implies the $\ell^2$-Liouville property as $C$ must then be $0$.
An analogous argument works if $L_c^{(-\infty)}$ is not essentially self-adjoint.
\end{proof}

We will now consider the transient case.
We first recall some basic facts concerning the 
construction of the Green's function.

We let $x_0 \in X$ and let $X_n$ denote a sequence of nested finite subsets of $X$
such that $\bigcup_n X_n = X$ and so that the subgraphs induced by $X_n$
are connected. We call such a sequence an \emph{exhaustion} of the graph.
We assume that $x_0 \in X_1$, recall that
$$\partial X_n = \{ x \in X_n \mid \textup{ there exists } y \not \in X_n \textup{ such that }
b(x,y)>0\} $$
denotes the boundary of $X_n$ and $\mathring{X}_n = X_n \setminus \partial 
X_n$ denotes the interior of $X_n$.
We then consider solutions to the equation
$$
\begin{cases}
\Delta g_n = \widehat{1}_{x_0} & \textup{ on } \mathring{X}_n \\
g_n = 0 & \textup{ on } \partial X_n
\end{cases}
$$
where $\widehat{1}_{x_0} = 1_{x_0}/m(x_0)$.
By some easy minimum principle arguments, we obtain 
$$0 < g_n \leq g_n(x_0) \qquad \textup{ and } \qquad g_n \leq g_{n+1}$$
on $\mathring{X}_n$.
Transience is then equivalent to the finiteness of the limit
as $n \to \infty$. If  the limit exists, then it is called the \emph{Green's function with pole at $x_0$}
which we denote by $g_{x_0}$ or by $g$, i.e., $g_n \to g$ whenever the graph
is transient.
In particular, 
$g$ satisfies
$$\De g = \widehat{1}_{x_0} \qquad \textup{ 
and } \qquad g(x_0) \geq g(x)$$ for all $x \in X$.
Further arguments give that $g(x_0) > g(x)$ for all $x \neq x_0$
and that the construction above is independent of the choice
of the exhaustion sequence.

In the case of birth-death chains, the Green's function takes a form that we have already
seen in Corollary~\ref{c:vanishing}.
\begin{example}[Green's function on birth-death chains]
Let $(\N_0,b,m)$ be a transient birth-death chain. Then, the Green's function with pole at $0$ takes the
form
$$g(r) = \sum_{k=r}^\infty \frac{1}{b(k,k+1)}.$$
\end{example}

We will next use the material above to characterize the
$\ell^2$-Liouville property on stars with two rays in the transient case.
%This completes the picture started in Proposition~\ref{p:recurrent_rays}. In the recurrent case, we have 
%$\sum_{k} 1/b(k,k+1)=\sum_{k}1/b(-k,-k-1) = \infty$ and thus by
%Lemma~\ref{l:double_harmonic} we have that all harmonic functions are unbounded
%at both $\infty$ and $-\infty$. 
Unlike the recurrent case where all harmonic functions are unbounded
at both ends, 
the transient case is a bit more involved as a star
with two rays is transient if and only if either $\sum_{k} 1/b(k,k+1)< \infty$ or
$\sum_{k}1/b(-k,-k-1) < \infty$ or both. Thus, it makes sense to split the
characterization into cases.

We will say that a star with two 
rays is \emph{transient at $\infty$} if 
$$\sum_{k=0}^\infty \frac{1}{b(k,k+1)}< \infty$$ and \emph{transient at $-\infty$} if
$$\sum_{k=0}^\infty \frac{1}{b(-k,-k-1)}< \infty.$$
Furthermore, we will say that the \emph{Green's function is in $\ell^2$ at infinity}, or $g \in \ell^2(\infty),$ if
$$\sum_{r=0}^\infty \left(\sum_{k=r}^\infty \frac{1}{b(k,k+1)} \right)^2 m(r) < \infty$$
and the \emph{Green's function is in $\ell^2$ at negative infinity}, or $g \in \ell^2(-\infty)$, if
$$\sum_{r=0}^\infty \left(\sum_{k=r}^\infty \frac{1}{b(-k,-k-1)} \right)^2 m(-r) < \infty.$$
Finally, we write $m(\infty) < \infty$ whenever $\sum_{r=0}^\infty m(r)< \infty$ and
$m(-\infty) < \infty$ whenever $\sum_{r=0}^\infty m(-r)< \infty.$

With these notations, we now start our characterization for the transient case.
First we consider the case of transience at both ends. In this case, all harmonic
functions are bounded and we only have to consider if they vanish at either end or not.
We illustrate this as well as the Green's function in Figure~\ref{fig:M3}.

\begin{figure} \centering
%\begin{minipage}{1\textwidth}
 \centering 
\begin{tikzpicture}[scale=0.50]
\begin{axis}[
xlabel={},
ylabel={},
xmin=-10, xmax=10,
ymin=-200, ymax=200,
 yticklabel=\empty,
  xticklabel=\empty,
]
\addplot+[y,mark=square*  ,mark options={solid,fill=blue}, domain=1:10, samples at={-0.8, -0.35, 0.35, 0.8},circle, blue] {atan(x) };
    \addplot+[y,mark=square*  ,mark options={solid,fill=blue}, domain=1:10, samples at={0,0.8,..., 10},circle, blue] {atan(x) };
     \addplot+[y,mark=square*,mark options={solid,fill=blue},  domain=1:10, samples at={0,-0.8,..., -10},circle, blue] {atan(x) };
      \addplot+[y ,mark=-,samples at={-10, -9.2,...,10},red] {-94};
       \addplot+[y ,mark=-,samples at={-10, -9.2,...,10},red] {94};
\end{axis}
\end{tikzpicture}  
\qquad
%\caption{Harmonic functions: transient at $\pm \infty$} \label{fig:M2}
%\end{minipage}\hfill
%\begin{minipage}{1\textwidth}
\centering
\begin{tikzpicture}[scale=0.50]
\begin{axis}
[
xlabel={},
ylabel={},
xmin=-10, xmax=10,
ymin=0, ymax=6,
 yticklabel=\empty,
  xticklabel=\empty,
]
\addplot+[y,mark=square*,mark options={solid,fill=blue}, domain=1:10, samples at={0,0.3, 0.8},circle, blue]{3*pow(2,-x)};
\addplot+[y,mark=square*,mark options={solid,fill=blue}, domain=1:10, samples at={0,-0.3, -0.8},circle, blue]{3*pow(2,x)};
\addplot+[y,mark=square*,mark options={solid,fill=blue}, domain=1:10, samples at={0.8,1.8,..., 10},circle, blue]{3*pow(2,-x)};
\addplot+[y,mark=square*,mark options={solid,fill=blue}, domain=1:10, samples at={-0.8,-1.8,..., -10},circle, blue]{3*pow(2,x)};
\end{axis}
\end{tikzpicture}
\caption{Harmonic (left) and Green's (right) functions: transience at $\pm \infty$} \label{fig:M3}
%\end{minipage}
\end{figure}

\begin{proposition}[Characterization of $\ell^2$-Liouville: transience at $\pm \infty$]\label{p:transient_both_ends}
Let $(\Z,b,m)$ be a star with two rays which is transient at both $\infty$ and $-\infty$.
Then, $(\Z,b,m)$ does not satisfy the $\ell^2$-Liouville property if and only if at least one of the following
conditions holds:
\begin{enumerate}
\item $g \in \ell^2(\infty)$ and $m(-\infty)< \infty$.
\item $g \in \ell^2(-\infty)$ and $m(\infty)< \infty$.
\item $m(\Z)<\infty$.
\end{enumerate}
\end{proposition}
\begin{proof}
Let $v$ be a non-constant harmonic function on $(\Z,b,m)$. By Lemma~\ref{l:double_harmonic} and the fact that
the graph is transient at both $\infty$ and $-\infty$, it follows that $v$ must be
bounded and monotonic. Thus, there are only the following possibilities for the behavior of $v$
at $\pm \infty$: 
\begin{itemize}
\item If $v(r) \to 0$ as $r \to \infty$ and $v(r) \to K \neq 0$ as $r \to -\infty$, then $v \in \ell^2(X,m)$
if and only if (1) holds by Corollary~\ref{c:vanishing}.
\item  If $v(r) \to 0$ as $r \to -\infty$ and $v(r) \to K \neq 0$ as $r \to \infty$, then $v \in \ell^2(X,m)$
if and only if (2) holds by Corollary~\ref{c:vanishing}.
\item If $v(r) \to K_1 \neq 0$ as $r \to \infty$ and $v(r) \to K_2 \neq 0$ as $r \to -\infty$, then $v \in
\ell^2(X,m)$ if and only if (3) holds.
\end{itemize}
This completes the proof.
\end{proof}
\begin{remark}
We note that (3) holds if and only if both (1) and (2) hold.
\end{remark}

In the other cases of transience, one has unbounded harmonic functions. 
We may visualize them and the Green's function as in Figure~\ref{fig:M5}.

\begin{figure} \centering
%\begin{minipage}{1\textwidth}
 \centering 
\begin{tikzpicture}[scale=0.50]
\begin{axis}
[
xlabel={},
ylabel={},
xmin=-50, xmax=50,
ymin=-100, ymax=100,
 yticklabel=\empty,
  xticklabel=\empty,
]
    
    \addplot+[y,mark=square*  ,mark options={solid,fill=blue}, domain=1:10, samples at={3.5,7.5,...,50},circle, blue] {-atan(x) +63 };
\addplot+[y, mark=square*,,mark options={solid,fill=blue}, domain=1:10, samples at={8,0,-4,...,-50}, circle, blue]{-1.9*x-4};
\end{axis}
\end{tikzpicture}
%\caption{Harmonic functions: transient at $\infty$} \label{fig:M4}
%\end{minipage}\hfill
\qquad
%\begin{minipage}{1\textwidth}
\centering
\begin{tikzpicture}[scale=0.50]
\begin{axis}
[
xlabel={},
ylabel={},
xmin=-10, xmax=10,
ymin=0, ymax=6,
 yticklabel=\empty,
  xticklabel=\empty,
]
\addplot+[y,mark=square*,mark options={solid,fill=blue}, domain=1:10, samples at={0,0.3, 0.8},circle, blue]{3*pow(2,-x)};
\addplot+[y,mark=square*,mark options={solid,fill=blue}, domain=1:10, samples at={0.8,1.8,..., 10},circle, blue]{3*pow(2,-x)};
\addplot+[y,mark=square*,mark options={solid,fill=blue}, domain=1:10, samples at={0.5,1,1.4},circle, blue]{3*pow(2,-x)};
\addplot+[y,mark=square*,mark options={solid,fill=blue}, domain=1:10, samples at={-0.8,-1.8,..., -10},circle, blue]{3};
\end{axis}
\end{tikzpicture}
\caption{Harmonic (left) and Green's (right) functions: transience at $\infty$} \label{fig:M5}
%\end{minipage}
\end{figure}

In this case,
we provide connections to essential self-adjointness. Here, we again denote by $L_{c}^{(\infty)}$
the Laplacian on the positive part of the graph and by $L_{c}^{(-\infty)}$ the Laplacian
on the negative part. We then have the following results.
\begin{proposition}[Characterization of $\ell^2$-Liouville: transience at $\infty$]\label{p:transient_one_end}
Let $(\Z,b,m)$ be a star with two rays which is transient only at $\infty$.
Then, $(\Z,b,m)$ does not satisfy the $\ell^2$-Liouville property if and only if
$L_c^{(-\infty)}$ is not essentially self-adjoint and at least one of the following holds:
\begin{enumerate}
\item $g \in \ell^2(\infty)$.
\item $m(\infty)<\infty$.
\end{enumerate}
\end{proposition}
\begin{proof}
Let $v$ be a non-constant harmonic function. By Lemma~\ref{l:double_harmonic} and the fact that
the graph is transient at only $\infty$, we have that $v$ must be monotonic, 
bounded at $\infty$ and unbounded at $-\infty$. 
By Theorem~\ref{t:star-like_characterization}, $L_c^{(-\infty)}$ is not essentially self-adjoint if and only if
$$\sum_{r=1}^\infty \left( \sum_{k=1}^r \frac{1}{b(-k,-k-1)}\right)^2 m(-r-1)
<\infty$$
which, by the limit comparison test, is equivalent to 
$$\sum_{r = 1}^\infty\left(v(-1) - C \sum_{k=1}^r \frac{1}{b(-k,-k-1)}\right)^2
m(-r-1)<\infty$$
where $C= b(0,1)(v(1)-v(0))\neq 0$ as $v$ is non-constant.
Now, the result follows by considering the cases of $v(r) \to 0$ as $r \to \infty$
or $v(r) \to K \neq 0$ as $r \to \infty$ and using Corollaries~\ref{c:vanishing}~and~\ref{c:v_square_sum}. This completes the proof.
\end{proof}

Finally, by symmetry, we obtain the following result.
\begin{proposition}[Characterization of $\ell^2$-Liouville: transience at $-\infty$]\label{p:transient_other_end}
Let $(\Z,b,m)$ be a star with two rays which is transient only at $-\infty$.
Then, $(\Z,b,m)$ does not satisfy the $\ell^2$-Liouville property if and only if
$L_c^{(\infty)}$ is not essentially self-adjoint and at least one of the following holds:
\begin{enumerate}
\item $g \in \ell^2(-\infty)$.
\item $m(-\infty)<\infty$.
\end{enumerate}
\end{proposition}

\begin{remark}[Characterization of $\ell^2$-Liouville]
We note that by combining Propositions~\ref{p:recurrent_rays},~\ref{p:transient_both_ends},
~\ref{p:transient_one_end},~and~\ref{p:transient_other_end}, we obtain
a full characterization of the $\ell^2$-Liouville in the case of a star with two rays.
\end{remark}

We finish this subsection by strengthening the result which states that for general graphs
the $\ell^2$-Liouville property along with infinite measure and strict positivity of the Laplacian
implies essential self-adjointness for the case of transient star-like graphs where 
all birth-death chains are rays. We will replace the infinite measure and strict positivity
assumptions by assuming that the Green's function is in $\ell^2$.

For this result to hold, we will need the fact that if the Green's function
is in $\ell^2$ for one pole, then it is in $\ell^2$ for all poles. This 
is shown in the next lemma.
\begin{lemma}\label{l:one_all}
Let $(X,b,m)$ be a transient graph. If $g_{x_0} \in \ell^2(X,m)$ for some
$x_0 \in X$, then $g_{x} \in \ell^2(X,m)$ for all $x \in X$.
\end{lemma}
\begin{proof}
Let $x \in X$ and let $\Omega \subseteq X$
be a finite set such that $x_0, x \in \Omega$ and so that $\Omega$
induces a connected subgraph. We denote
the corresponding Green's functions with poles at $x_0$ and $x$ by
$g_{x_0}$ and $g_{x}$, respectively. We choose an exhaustion $(X_n)$ such that $\Omega \subseteq X_1$
and let $g_{x_0,n}$ and $g_{x,n}$ denote the sequences of Green's functions
associated to the exhaustion.

As $g_{x_0,n}, g_{x,n} > 0$ on $\Omega$ and $\Omega$ is finite, we can choose a constant $C>0$ such that
$$ C g_{x_0,1} \geq g_{x} \qquad \textup{ on } \Omega.$$
By the monotonicity of the Green's functions in the exhaustion sequences, this implies
$$ C g_{x_0,n} \geq g_{x,n} \qquad \textup{ on } \Omega$$
for all $n \in \N$. Let $u_n = C g_{x_0,n} - g_{x,n}$ and observe that $u_n$ satisfies
$$
\begin{cases}
\Delta u_n = 0 & \textup{ on } \mathring{X}_n \setminus \Omega \\
u_n = 0 & \textup{ on } \partial X_n \\
u_n \geq 0 & \textup{ on } \partial \Omega
\end{cases}
$$
which then implies $u_n \geq 0$ on $X_n \setminus \Omega$ by a minimum principle.
Taking the limit as $n \to \infty$ now gives $C g_{x_0} \geq g_{x}$ and
thus $g_{x} \in \ell^2(X,m)$ as we assume that $g_{x_0} \in \ell^2(X,m)$.
This completes the proof.
\end{proof}

Thus, in case that $g_x \in \ell^2(X,m)$ for some $x \in X$, we may simply write
$g \in \ell^2(X,m)$.
We then have the following result.
\begin{theorem}\label{t:transient_Liouv}
Let $(X,b,m)$ be a transient star-like graph where all birth-death chains are rays. 
If $(X,b,m)$ satisfies the $\ell^2$-Liouville property and $g \in \ell^2(X,m)$, 
then $L_c$ is essentially self-adjoint.
\end{theorem}
\begin{proof}
Suppose that $L_c$ is not essentially self-adjoint and $g \in \ell^2(X,m)$. 
By Theorem~\ref{t:star-like_characterization}, this means that $L_c^{(i)}$ is not essentially self-adjoint
for some $i \in I$ and by reordering we can assume that $L_c^{(1)}$
is not essentially self-adjoint where $L_c^{(1)}$ is the Laplacian on the birth-death chain
$(\N_0^{(1)}, b, m)$ which is a ray of the graph by assumption.
Let $x_0$ be the unique vertex in $X_1$ that is connected to $(\N_0^{(1)}, b, m)$ and
let $g$ be the Green's function with pole at $x_0$. By assumption 
and Lemma~\ref{l:one_all}, $g \in \ell^2(X,m)$. 

We will define a non-constant
harmonic function $v \in \ell^2(X,m)$.
We start by letting $v(x)=g(x)$ for all $x \not \in \N_0^{(1)}$.
Now, at the pole $x_0$, we make $v$ harmonic by letting
$$v(0^{(1)})= g(x_0)+ \frac{1}{b(x_0,0^{(1)})} \sum_{y \in X_1} b(x_0,y)(g(x_0)-g(y)).$$
We then extend $v$ to be harmonic on the ray by following Lemma~\ref{l:double_harmonic}
and letting
$$v((r+1)^{(1)}) = v(0^{(1)}) + C \sum_{k=0}^{r-1} \frac{1}{b(k^{(1)},(k+1)^{(1)})}$$ 
for all $r \geq 0$ where $C=b(x_0,0^{(1)})(v(0^{(1)})-g(x_0))$.
We note that both $v(0^{(1)})>0$ and $C>0$
since $g(x_0) > g(x)$ for all $x \neq x_0$.
The fact that $v$ is harmonic follows by direct calculations. 
The fact that $v \in \ell^2(X,m)$ follows from the assumptions that
$g \in \ell^2(X,m)$, that $L_c^{(1)}$ is not essentially self-adjoint, 
Corollary~\ref{c:v_square_sum} and basic arguments.
\end{proof}

\begin{remark}
As mentioned previously, the $\ell^2$-Liouville property along with $\lm_0(L)>0$ and $m(X)=\infty$
implies essential self-adjointness for general graphs. 
As $\lm_0(L)>0$ implies both transience and the fact that $g \in \ell^2(X,m)$,
see \cite{HMW21}, we see that the above
is an improvement in the case of star-like graphs with rays.
\end{remark}

\subsection*{Acknowledgments}
The authors are grateful to Matthias Keller, Aleksey Kostenko, Ognjen Milatovic, Noema Nicolussi,
and Marcel Schmidt for helpful comments.

\end{document}